\documentclass[10pt,reqno,a4paper]{amsart}

\usepackage{amsmath}
\usepackage{enumerate}
\usepackage{amssymb}
\usepackage{amscd}
\usepackage{amsthm}
\usepackage{amsfonts}
\usepackage{graphicx}
\usepackage{hyperref}

\usepackage{ccfonts,eulervm} 
\usepackage[T1]{fontenc}

\usepackage[utf8]{inputenc}
\numberwithin{equation}{section}

\newtheorem{theorem}{Theorem}[section]

\newtheorem{corollary}{Corollary}[section]
\newtheorem{proposition}{Proposition}[section]
\newtheorem{lemma}{Lemma}[section]

\theoremstyle{definition}

\newtheorem{remark}[theorem]{Remark}

\newtheorem{problem}{Problem}

\newcommand{\cB}{\mathcal{B}}

\newcommand{\cH}{\mathcal{H}}

\newcommand{\cK}{\mathcal{K}}

\newcommand{\cM}{\mathcal{M}}

\newcommand{\cP}{\mathcal{P}}

\newcommand{\cR}{\mathcal{R}}

\newcommand{\bF}{\mathbb{F}}

\newcommand{\bP}{\mathbb{P}}

\newcommand{\bR}{\mathbb{R}}

\newcommand{\bZ}{\mathbb{Z}}

\newcommand{\ra}{\rightarrow}

\newcommand{\qand}{\quad \textrm{and} \quad}

\newcommand\subsetsim{\mathrel{%
\ooalign{\raise0.2ex\hbox{$\subset$}\cr\hidewidth\raise-0.8ex\hbox{\scalebox{0.9}{$\sim$}}\hidewidth\cr}}}
\newcommand{\eps}{\varepsilon}
\newcommand{\cmu}{\check{\mu}}

\DeclareMathOperator{\linspan}{span}
\DeclareMathOperator{\Aff}{Aff}

\DeclareMathOperator{\supp}{supp}

\DeclareMathOperator{\esssup}{ess \, sup}

\begin{document}

\title{Five remarks about random walks on groups}

\author{Michael Bj\"orklund}
\address{Department of Mathematics, Chalmers, Gothenburg, Sweden}
\email{micbjo@chalmers.se}
\thanks{}



\date{}

\dedicatory{}

\begin{abstract}
The main aim of the present set of notes is to give new, short and essentially
self-contained proofs of some classical, as well as more recent, 
results about random walks on groups. For instance, we shall 
see that the drift characterization of Liouville groups, due to Kaimanovich-Vershik and Karlsson-Ledrappier (and to Varopoulos in some important special cases) admits a very short and quite elementary proof. Furthermore, we give a new, and rather short proof of (a weak version of) an observation of Kaimanovich (as well as a small strengthening thereof) that the Poisson boundary of any symmetric measured group $(G,\mu)$, is doubly ergodic, and the diagonal $G$-action on its product is ergodic with unitary coefficients. We also offer a characterization of weak mixing for 
ergodic $(G,\mu)$-spaces parallel to the measure-preserving case.

We shed some new light on Nagaev's classical technique to prove central limit
theorems for random walks on groups. In the interesting special
case when the measured group admits a product current, we  
define a Besov space structure on the space of bounded harmonic 
functions with respect to which the the associated convolution operator 
is quasicompact without any assumptions on finite exponential moments. 
For Gromov hyperbolic measured groups, this gives an alternative proof of the fact that every H\"older continuous function with zero integral with respect to the unique stationary probability measure on the Gromov boundary is a co-boundary.

Finally, we give a new and almost self-contained proof of a 
special case of a recent combinatorial result about piecewise
syndeticity of product sets in groups by the author and A. Fish.
\end{abstract}

\maketitle

\section{Drifts of random walks and the Liouville property}
The study of random walks on countable groups is to a large extent concerned with the asymptotic behavior of convolution powers of 
some fixed probability measures on the groups. One is particularly
interested in the growth of the integrals of certain geometrically 
defined functions against these convolution powers. For instance, 
let $G$ be a countable group and let $\mu$ be a probability 
measure on $G$ with the property that the support of $\mu$ 
generates $G$ as a semigroup. We shall refer to $(G,\mu)$ as a
\emph{measured group}. Given a left $G$-invariant and $\mu$-integrable 
metric $d$ on $G$, we define the \emph{drift} by
\[
\ell_d(\mu) = \lim_n \frac{1}{n} \int_G d(g,e) \, d\mu^{*n}(g),
\]
where $\mu^{*n}$ denotes the $n$-th convolution power of $\mu$. An
elementary sub-additivity argument (Fekete's Lemma) guarantees that the
limit exists and is finite. From a probabilistic point of view, the drift 
$\ell_d(\mu)$ measures the asymptotic linear speed (with respect to $d$) of a sequence of products of independent and $\mu$-distributed elements in $G$. Since metrics are symmetric functions on $G \times G$, 
we always have $\ell_d(\cmu) = \ell_d(\mu)$, where $\cmu(g) = \mu(g^{-1})$ for all $g$ in $G$. We say that $\mu$ is a \emph{symmetric} probability measure (and $(G,\mu)$ is a \emph{symmetric measured group}) if $\cmu = \mu$.

We note that if $G$ is generated by a finite (symmetric) set $S$, then the word metric $d_S$ with respect to $S$ has the property that for every left $G$-invariant metric $d$, there exists a constant $C$ such that 
the inequality $d(g,e) \leq C \cdot d_S(g,e)$ holds for all $g \in G$, so in particular, if $\ell_{d_S}(\mu) = 0$ for some probability measure $\mu$, 
then $\ell_d(\mu) = 0$ as well. 

Our aim in this section is to give a characterization of those finitely 
generated and symmetric measured groups $(G,\mu)$ with 
$\ell_d(\mu) = 0$ for \emph{some} (and hence any) word-metric 
on $G$ in terms of bounded left $\mu$-harmonic functions. Recall
that a $\mu$-integrable complex-valued function $f$ on $G$ is 
\emph{left $\mu$-harmonic} if it satisfies
\begin{equation*}
\label{defharmonic}
(\cmu * f)(g) = \int_G f(sg) \, d\cmu(s), \quad \forall \, g \in G,
\end{equation*}
and it is \emph{right $\mu$-harmonic} if 
\begin{equation*}
\label{defharmonicright}
(f * \mu)(g) = \int_G f(gs) \, d\mu(s), \quad \forall \, g \in G.
\end{equation*}
If $f$ is both right \emph{and} left $\mu$-harmonic, then we say that 
$f$ is \emph{bi-$\mu$-harmonic}. Of special interest to us are the 
\emph{bounded} left $\mu$-harmonic functions on $G$. Clearly, 
constant functions on $G$ are left $\mu$-harmonic for every choice
of $\mu$, and we say that $(G,\mu)$ is (left) \emph{Liouville} if there 
are no non-constant left $\mu$-harmonic functions. Since the function 
$\check{f}(g) = f(g^{-1})$ is right $\mu$-harmonic if and 
only if $f$ is left-$\mu$-harmonic, we see that the notions of left and 
right Liouville coincide. 

The original proof of the following theorem combined a series of fundamental observations of Avez \cite{A72}, Derriennic \cite{D80}, Kaimanovich-Vershik \cite{KV83} and Karlsson-Ledrappier \cite{KL07} respectively. In the special case when $\mu$ is finitely supported, Varopoulos established this theorem in \cite{V85}. We shall give a short proof of the general theorem below.

\begin{theorem}[A geometric characterization of measured Liouville groups]
\label{drift}
Let $(G,\mu)$ be a finitely generated and symmetric measured group 
and suppose $d$ is a word metric with respect to a finite symmetric 
generating set for $G$. Then $\ell_d(\mu) = 0$ if and only if $(G,\mu)$ is 
Liouville. 
\end{theorem}

\begin{remark}
Recall that a countable group $G$ is \emph{amenable} if every action of 
$G$ by homeomorphims on a compact hausdorff space $X$ admits a 
$G$-invariant probability measure. It is not hard to see that finite groups
and the group of integers are amenable, and that the class of amenable 
groups is closed under extensions and direct unions, which immediately
shows that every finite extension of a solvable group is amenable. Furthermore, every finitely generated group of sub-exponential growth 
can be shown to be amenable, while free groups on at least two generators,
and countable supergroups thereof are non-amenable. 

Suppose $(G,\mu)$ is a countable \emph{non-amenable} measured group
and let $X$ denote a compact Hausdorff space, equipped with an action of 
$G$ by homeomorphisms with no $G$-invariant probability measures. A 
simple application of Kaktuani's fixed point argument shows that there is 
always a probability measure $\nu$ on $X$ which satisfies the equation
\[
\int_G \int_X \phi(g^{-1} x) \, d\nu(x) d\mu(g) = \int_X \phi(x) \, d\nu(x) 
\]
for all $\phi \in C(X)$. Since, by assumption, $\nu$ cannot be $G$-invariant, 
there exists at least one $\phi \in C(X)$ such that the function
\[
f(g) = \int_X \phi(g^{-1}x) \, d\nu(x), \quad g \in G,
\]
is \emph{not} constant. It is readily verified that $f$ is left $\mu$-harmonic
and thus $(G,\mu)$ is \emph{not} Liouville. In particular, in view of Theorem 
\ref{drift}, if $G$ is a finitely generated \emph{non-amenable} group, then $\ell_d(\mu) > 0$ for \emph{every} word-metric $d$ on $G$ and (symmetric) probability measure $\mu$.
\end{remark}

\subsection{Liouville implies zero drift}
\label{subsecLD}
The proof of Theorem \ref{drift} splits naturally into two parts. The first one 
concerns the "if"-direction, for which the relevant result can be stated as 
follows. 

\begin{theorem}[Karlsson-Ledrappier \cite{KL07}]
\label{KL}
If $(G,\mu)$ is a measured Liouville group and $d$ is a left invariant 
and $\mu$-integrable metric $d$ on $G$, then there exists a 
real-valued and $\mu$-integrable homomorphism $u$ on $G$ such that
\[
\ell_d(\mu) = \int_G u(g) \, d\cmu(g).
\]
In particular, if $\mu$ is a symmetric probability measure on $G$, then $\ell_d(\mu) = 0$ for every left invariant and $\mu$-integrable metric $d$ on $G$.
\end{theorem}

Let us now outline a simple proof of this theorem, which has some similarities with the 
arguments in the paper \cite{EK10}, where are more quantitative version of Theorem 
\ref{KL} is proved. We first note that by the triangle-inequality, the sequence
\[
n \mapsto \int_G \big( d(g,x) - d(x,e) \big) \, d\mu^{*k}(x) 
\]
is bounded for every fixed $g$ in $G$, so by a simple diagonal argument, there exists a sub-sequence $(n_j)$ such that the limit 
\[
u(g) = \lim_{j \ra \infty} \frac{1}{n_j} \sum_{k=0}^{n_j-1}
\int_G \big( d(g,x) - d(x,e) \big) \, d\mu^{*k}(x) 
\]
exists for all $g \in G$ and thus, 
\[
\int_G u(sg) \, d\cmu(s) = u(g) + \int_G u(s) \, d\cmu(s), \quad \forall \, g \in G.
\]
We shall refer to functions $u$ with this property as \emph{left quasi-$\mu$-harmonic}. By dominated convergence, we have
\begin{eqnarray*}
\int_G u(s) \, d\cmu(s) 
&=&
\lim_{j \ra \infty} \frac{1}{n_j} \sum_{k=0}^{n_j-1}
\Big(
\int_G d(x,e) \, d\mu^{*(k+1)}(x) - \int_G d(x,e) \, d\mu^{*k}(x)  
\Big) \\
&=& 
\lim_{j \ra \infty} \frac{1}{n_j} 
\int_G d(x,e) \, d\mu^{*n_j}(x)
=
\ell_d(\mu).
\end{eqnarray*}
Furthermore, the triangle inequality guarantees that  the function $u$ is left 
Lipschitz, i.e. for every element $g$ in $G$, we have
\[
\sup_{x \in G} |u(xg) - u(x)| < + \infty. 
\]
The theorem of Karlsson-Ledrappier is now an immediate consequence of 
the following proposition, which is interesting in its own right.

\begin{proposition}
If $(G,\mu)$ is a Liouville measured group, then every left Lipschitz 
and left $\mu$-quasi-harmonic function on $G$ which vanishes at 
the identity is a homomorphism. 
\end{proposition}

\begin{proof}
Note that if $u$ is a left Lipschitz and left $\mu$-quasi-harmonic 
function on $G$, then for every $g \in G$, the function
\[
v_g(x) = u(xg) - u(x), \quad x \in G,
\]
is a \emph{bounded} left $\mu$-harmonic function on $G$, and 
hence constant. Since $u(e) = 0$, we conclude that
\[
u(xg) - u(x) = u(g)
\]
for all $g, x \in G$, that is to say, $u$ is a homomorphism.
\end{proof}

\begin{remark}
In particular, a Liouville group $(G,\mu)$ without homomorphisms 
into the additive group of the real numbers cannot admit any non-constant
\emph{left Lipschitz} and \emph{left (quasi) $\mu$-harmonic} functions. 
However, we shall see in Appendix I, that \emph{every} countable symmetric measured group $(G,\mu)$ always admits a non-constant \emph{left Lipschitz} and
\emph{right $\mu$-harmonic} function. 
\end{remark}

\subsection{Zero drift implies Liouville}

We now tend to the proof of the "only if"-direction in Theorem \ref{drift}, for
which the relevant result can be stated as follows. 

\begin{theorem}[Avez, Derriennic, Kaimanovich-Vershik, weaker form]
\label{KV}
Let $(G,\mu)$ be a finitely generated measured group and suppose $d$ is 
a left-invariant word metric on $G$ with respect to a finite symmetric generating set. If $\ell_d(\mu) = 0$, then $(G,\mu)$ is Liouville. 
\end{theorem}
\begin{remark}
The classical route to this theorem employs the entropy theory of measured
groups. One first shows that if $\ell_d(\mu) < +\infty$ for some (and hence
any) word metric $d$ on $G$, then the limit 
\[
h(G,\mu) 
= 
\lim_{n \ra \infty}
-\frac{1}{n} \sum_{g \in G} \mu^{*n}(g) \cdot \log \, \mu^{*n}(g) 
\]
exists and is finite. We refer to $h(G,\mu)$ as the (Avez) \emph{entropy}
of the measured group $(G,\mu)$, and the main result of Kaimanovich-Vershik in \cite{KV83} asserts (under the assumption that $\mu$ is symmetric and the Avez entropy is finite) that $h(G,\mu) = 0$ if and only if $(G,\mu)$ is 
Liouville. One is thus left with the task of showing that $\ell_d(\mu) = 0$ implies $h(G,\mu) = 0$. This is taken care of what is sometimes referred to as the
"fundamental" inequality (see e.g. Section 4 in \cite{KL07}), which we now formulate. Let $S$ be a finite and symmetric generating set for $G$ and let $d$ be the word metric associated to $S$. If one defines the \emph{exponential volume growth} 
of $(G,S)$ by
\[
v(G,S) = \varlimsup_{n \ra \infty} \frac{\log |S^n|}{n} \leq \log|S|
\]
then
\begin{equation}
\label{fund}
h(G,\mu) \leq v(G,S) \cdot \ell_d(\mu) = 0,
\end{equation}
which finishes the (classical) proof of Theorem \ref{KV}. Note that the
argument gives a bit more, namely that if $G$ has subexponential 
growth, i.e. $v(G,S) = 0$, then $h(G,\mu) = 0$ and thus $(G,\mu)$
is Liouville. By Theorem \ref{KL} and the remark following its statement, 
we conclude that $G$ is amenable and $\ell_d(\mu)  = 0$. 
\end{remark}

We shall now give a new alternative (and self-contained) proof of Theorem \ref{KV} which avoids the use of entropy theory. Let $(G,\mu)$ be a 
measured group and denote by $\cH^\infty_l(G,\mu)$ the space of all bounded left $\mu$-harmonic functions on $G$. We say that a Borel
probability space $(X,\nu)$ is a \emph{$(G,\mu)$-space} if $X$ is equipped 
with an action of $G$ by bi-measurable maps, which all preserve the
measure class of $\nu$, such that
\[
\int_G \int_X \phi(g^{-1}x) \, d\nu(x) \, d\mu(g) = \int_X \phi(x) \, d\nu(x)
\]
for all $\phi \in L^\infty(X,\nu)$. Probability measures with this property
are often referred to as \emph{$\mu$-harmonic} (or \emph{$\mu$-stationary}). Given a 
$\mu$-harmonic probability measure $\nu$ on $X$, one readily checks that the association 
\[
P_\nu\phi(g) = \int_X \phi(g^{-1}x) \, d\nu(x), \quad g \in G.
\]
defines an element in $\cH^\infty_l(G,\mu)$ for every $\phi \in L^\infty(X,\nu)$. We note that if $\nu$ is $G$-invariant, then such elements are all 
constants. A remarkable fact, often attributed to Furstenberg, is that every
measured group $(G,\mu)$ admits a $(G,\mu)$-space $(B,m)$, which we
shall refer to to as the \emph{Poisson boundary of $(G,\mu)$}, 
for which the linear map $P_m$ above is in fact \emph{isometric and onto} 
$\cH^\infty_l(G,\mu)$. There are many constructions of the Poisson boundary
of a measured group in the literature. We refer the reader to \cite{Fu} for a 
detailed exposition of one of the more elementary constructions.

\begin{proposition}[Furstenberg]
\label{furstenberg}
For every measured group $(G,\mu)$ there exists an ergodic $(G,\mu)$-space $(B,m)$, which we shall refer to as the Poisson boundary of $(G,\mu)$, such that the Poisson transform $P : L^\infty(B,m) \ra \cH_l^{\infty}(G,\mu)$ defined by
\[
P\phi(g) = \int_B \phi(g^{-1}b) \, dm(b), \quad g \in G,
\] 
is an isometric isomorphism. In particular, $(B,m)$ is trivial if and only
if $m$ is $G$-invariant.
\end{proposition}

\begin{remark}
As was pointed out by Jaworski in \cite{Ja94}, the Poisson boundary $(G,m)$
is strongly approximately transitive (SAT), i.e. for every measurable subset 
$A \subset B$ of positive $m$-measure and for every $\eps > 0$, there 
exists $g \in G$ such that $m(gA) > 1 - \eps$. Indeed, since $P$ is 
isometric, if $A \subset B$ has positive $m$-measure, then
\[
\sup_{g \in G} P\chi_A(g) = \sup_{g \in G} \int_B \chi_A(g^{-1}b) \, dm(b) = 
\sup_{g \in G} m(gA) = 1,
\]
from which the SAT-property follows.
\end{remark}

Since the $G$-action on $(B,m)$ preserves the measure class of $m$
(i.e. the set of all null-sets for $m$), the Radon-Nikodym derivative 
\[
\sigma_m(g,b) = \frac{dg^{-1}m}{dm}(b)
\]
is a well-defined non-negative element in $L^1(B,m)$ for every $g \in G$, and one
readily checks that
\begin{equation}
\label{multcocycle}
\sigma_m(g_1g_2,b) =  \sigma_m(g_1,b) \, \sigma_m(g_2,g_1b)
\end{equation}
for all $g_1, g_2 \in G$ and for almost every $b$ with respect to $m$, 
that is to say $\sigma_m$ is a multiplicative cocycle for the $G$-action
on $(B,m)$. A crucial feature with this cocycle is its $\mu$-harmonicity, namely
\begin{equation}
\label{cocycleharm}
\int_G \sigma_m(g,b) \, d\mu^{*n}(g) = 1, \quad \textrm{for a.e. $b$ and for all $n \geq 1$}.
\end{equation}
Indeed, for every $\phi \in L^\infty(B,m)$, one has
\[
\int_B \Big( \int_G \sigma_m(g,b) \, d\mu^{*n}(g) \Big) \, \phi(b) \, dm(b) 
= \int_G \int_B \phi(g^{-1}b) \, dm(b) \, d\mu^{*n}(g) = \int_B \phi \, dm,
\]
for all $n$, which immediately yields \eqref{cocycleharm}. \\

After a simple use of Jensen's inequality and \eqref{cocycleharm}, Theorem 
\ref{KV}  can now be reformulated as follows.

\begin{theorem}[Kaimanovich-Vershik, weak version]
Let $(G,\mu)$ be a finitely generated measured group with Poisson boundary $(B,m)$ and suppose $\ell_d(\mu) = 0$ for some (and hence any) word-metric with respect to a finite symmetric generating set. Then $m$ is 
$G$-invariant, or equivalently
\[
\int_G \int_B \log \frac{dg^{-1}m}{dm}(b) \, dm(b) \, d\mu(g) = 0.
\]
\end{theorem}

We note that equation \eqref{cocycleharm} implies that $\sigma_m(g,\cdot)$ is not only in $L^1(B,m)$
for every $g$, but is in fact essentially bounded. This can be seen as
follows. Since we assume that the support of $\mu$ generates $G$ as
a semigroup, there exists for every $s$ in $G$, an integer $n$ such that
the measure $\mu^{*n}(s)$ is positive, and thus
\[
\sigma_m(s,b) \mu^{*n}(s) \leq \int_G \sigma_m(g,b) \, d\mu^{*n}(g) = 1
\]
for $m$-almost every $b$ in $B$. We conclude that 
\[
\|\sigma_m(s,\cdot)\|_\infty \leq \frac{1}{\mu^{*n}(s)} < +\infty, 
\]
where $n$ is chosen as above, and 
\[
\sigma_m(s,\cdot) \geq \frac{1}{\sigma_m(s^{-1},s \cdot )} \geq \frac{1}{\|\sigma_m(s^{-1},\cdot)\|_\infty}, \quad \forall \, s \in G.
\]
In particular, if we define 
\[
c_m(g,b) = -\log \sigma_m(g,b),
\]
then one can think of $c_m$ as a map from $G$ into $L^\infty(B,m)$, which
satisfies the equations
\[
c_m(g_1 g_2) = c_m(g_1) +g_1^{-1} c_m(g_2), \quad \forall \, g_1, g_2 \in G,
\]
where $G$ acts on $L^\infty(B,m)$ via the left regular representation. We shall refer to such maps $c$ from $G$ into $L^\infty(B,m)$ as \emph{cocycles}, and one readily checks that if $c$ is a cocycle, then 
\[
\rho_c(g) = \|c(g)\|_\infty, \quad g \in G,
\]
defines a semi-metric on $G$, i.e. 
\[
\rho_c(g_1,g_2) \leq \rho_c(g_1) + \rho_c(g_2), \quad \forall \, g_1, g_2 \in G.
\]
We reserve the notation $\rho_m$ for the semi-metric associated to the cocycle $c_m$ above and refer to $\rho_m$ as the \emph{canonical semi-metric on $G$ associated with $(G,\mu)$}. 

\begin{remark}
In order to get a better feeling for the canonical semi-metric of a measured group, let us
consider the case when $G = \bF_2$, the free group on two free generators $a$ and $b$,
equipped with the (symmetric) probability measure
\[
\mu = \frac{1}{4}\big( \delta_a + \delta_b + \delta_{a^{-1}} + \delta_{b^{-1}} \big).
\]
Let $\partial \bF_2$ denote the compact space of all infinite one-sided reduced words in $a$ and $b$ and their inverses and note that the action of $\bF_2$ on itself extends to an
action by homeomorphisms on $\partial \bF_2$. One readily checks that the 
Borel probability measure $m$ on $\partial \bF_2$ which assigns the same measure to all
cylinder sets in $\partial \bF_2$ corresponding to words of the same word length is 
$\mu$-harmonic, and one can prove that $(\partial \bF_2,m)$ realizes the Poisson 
boundary for $(\bF_2,\mu)$. Furthermore, the Radon-Nikodym cocycle of $m$ is given
by
\[
\sigma_m(g,\xi) = 3^{-(\|g\|-2(g,\xi))}, \quad (g,\xi) \in \bF_2 \times \partial \bF_2,
\]
where $\|\cdot\|$ denotes the word metric on $\bF_2$ (with respect to $a$ and $b$
and their inverses) and $(g,\xi)$ is the length of the longest common sub-word of 
$g$ and $\xi$ (this is often referred to as the \emph{confluent} or \emph{Gromov product} in the literature). A straightforward calculation now yields
\[
\rho_m(g) = \|g\| \cdot \ln 3, \quad \forall \, g \in G,
\]
which in particular shows that in this special case, $\rho_m$ is in fact a metric. 
\end{remark}

The following simple proposition relates the asymptotic behavior of a cocycle $c$ to the vanishing of the drift of $(G,\mu)$ with respect to word-metrics on finitely generated groups, and finishes the proof of our version of the theorem of Kaimanovich-Vershik.

Recall that a \emph{$\mu$-harmonic mean $\lambda$ on $L^\infty(B,m)$} is a functional on $L^\infty(B,m)$ which is positive, i.e. gives non-negative
values to non-negative elements in $L^\infty(B,m)$, normalized, i.e. 
$\lambda(1) = 1$ and satisfies
\[
\int_G \lambda(g \cdot \phi) \, d\mu(g) = \lambda(\phi), \quad \forall \, \phi \in L^\infty(B,m)
\]
where $G$ acts on $L^\infty(B,m)$ via the left regular representation. In 
particular, the measure $m$ is a $\mu$-harmonic mean on $L^\infty(B,m)$.

\begin{proposition}
Let $(G,\mu)$ be a finitely generated measured group such that 
$\ell_d(\mu) = 0$ for some left invariant word-metric $d$ with 
respect to a finite symmetric generating set of $G$. Let $(B,m)$ 
denote the Poisson boundary of $(G,\mu)$. If $c : G \ra  L^\infty(B,m)$
is a cocycle, then 
\[
\int_G \lambda(c(g)) \, d\mu(g) = 0
\]
for every $\mu$-harmonic mean $\lambda$ on $L^\infty(B,m)$. 
In particular, we have
\[
\int_B \int_G \log \frac{dg^{-1}m}{dm}(b) \, d\mu(g) dm(b) = 0,
\]
so $m$ is $G$-invariant, and thus $(B,m)$ is trivial.
\end{proposition}

\begin{proof}
Since $G$ is finitely generated, there exists for every cocycle $c : G \ra L^\infty(B,m)$ a constant $C_c$ such that
\[
\rho_c(g) \leq C_c \cdot d(g,e), \quad \forall \, g \in G,
\]
where $d$ is a word-metric on $G$ with respect to a finite symmetric
generating set of $G$. We assume that $\ell_d(\mu) = 0$, and thus
\[
\lim_{n} \frac{1}{n} \int_G \rho_c(g) \, d\mu^{*n}(g) = 0
\]
for every cocycle $c$. If $\lambda$ is a $\mu$-harmonic mean on 
$L^\infty(B,m)$, one readily checks that
\[
\int_G \lambda(c(g)) \, d\mu^{*n}(g) = n \cdot \int_G \lambda(c(g)) \, d\mu(g)
\]
for all $n$, and thus
\[
\Big| \int_G \lambda(c(g)) \, d\mu(g) \Big|
= 
\Big|\frac{1}{n} \int_G \lambda(c(g)) \, d\mu^{*n}(g) \Big|
\leq 
\frac{1}{n} \int_G \| c(g)\|_\infty \, d\mu^{*n}(g) \ra 0,
\]
which finishes the proof.
\end{proof}

\section{Ergodicity with unitary coefficients}

We now turn to some ergodic-theoretical aspects of random walks on 
groups. As we have seen, to every measured group $(G,\mu)$ one can
associate an ergodic $(G,\mu)$-space $(B,m)$, called the Poisson 
boundary of $(G,\mu)$ with the remarkable property that the linear
map $P : L^\infty(B,m) \ra \cH^\infty_l(G,\mu)$ defined by
\[
P\phi(g) = \int_B \phi(g^{-1}b) \, dm(b), \quad g \in G,
\]
is an isometric isomorphism. The aim of this section is to give a short
proof of (a weak version) of an observation of Kaimanovich in \cite{K03}, 
that the mere fact that this is an isomorphism onto 
$\cH_l^\infty(G,\mu)$ automatically forces significantly stronger ergodicity properties. 

\begin{theorem}[Kaimanovich, weak version]
\label{Kde}
Let $(G,\mu)$ be a measured group and denote by $(B,m)$ and 
$(\check{B},\check{m})$ the Poisson boundaries of $(G,\mu)$
and $(G,\cmu)$ respectively. If $(Y,\eta)$ is any ergodic probability
measure preserving $G$-space, then the diagonal action on the
triple $(B \times \check{B} \times Y, m \otimes \check{m} \otimes \eta)$
is ergodic.
\end{theorem}

\begin{remark}
It is not hard to show (see for instance the recent survey by Glasner-Weiss \cite{GW13}) that the
theorem above can be equivalently formulated as follows: For every 
unitary $G$-representation $(\cH,\pi)$ on a Hilbert space $\cH$, any
measurable $G$-equivariant map $F : B \times \check{B} \ra \cH$
must be essentially constant. Kaimanovich proves in \cite{K03} the
a priori stronger statement that one can assert the same thing about 
$G$-equivariant and weak*-measurable maps from $B \times \check{B}$
into \emph{any} isometric $G$-representation on any (separable) Banach space. For certain applications in bounded cohomology, this seemingly
stronger statement is needed. Since we wish to keep the discussions in 
this paper fairly short, we shall confine ourselves to the setting of Theorem \ref{Kde}, although many of the techniques we shall describe can be used 
to give a complete proof of the main result in \cite{K03}.
\end{remark}

To start the proof of Theorem \ref{Kde}, we first observe that if $(\check{B},\check{m})$ is the Poisson boundary of 
$(G,\cmu)$ and $(Y,\eta)$ is any probability measure preserving $G$-space,
then the diagonal $G$-action on $(\check{B} \times Y,m \otimes \eta)$ is 
a $(G,\cmu)$-space. Hence Theorem \ref{Kde} follows immediately from 
the following result, which we have not been able to directly locate in the
literature. 

\begin{theorem}
\label{K}
Let $(G,\mu)$ be a measured group with Poisson boundary $(B,m)$
and suppose $(X,\nu)$ is an ergodic $(G,\cmu)$-space. Then the 
diagonal $G$-action on the product space $(B \times X,m \otimes \nu)$ 
is ergodic. 
\end{theorem}

\begin{remark}
In order to see how Theorem \ref{Kde} follows this statement, we argue in two steps. 
First note that if $(Y,\eta)$ is an \emph{ergodic} probability measure preserving $G$-space, then 
it is an ergodic $(G,\cmu)$-space as well, and  Theorem \ref{K} implies that $X = \check{B} \times Y$, with the probability measure $\nu = \check{m} \otimes \eta$, is an \emph{ergodic} $(G,\cmu)$-space. If we now apply Theorem \ref{K} to the diagonal $G$-action on the direct product
\[
(B \times X,m \otimes \nu) = 
(B \times \check{B} \times Y,m \otimes \check{m} \otimes \eta),
\]
then we conclude that it is also ergodic, which is exactly the assertion of 
Theorem \ref{Kde}. \\

In the case when $(X,\nu)$ is an ergodic probability measure preserving 
$G$-space, Theorem \ref{K} is due to Aaronson and Lema\'nczyk in 
\cite{AL05}. Their proof however follows quite different lines than ours. 
\end{remark}

We now begin our proof of Theorem \ref{K}. Let $(X,\mu)$ be an ergodic $(G,\cmu)$-space, and suppose $f$ is a
$G$-invariant essentially bounded function on $B \times X$, which we
may without loss of generality assume to have zero integral. We wish to
prove that $f$ vanishes identically. 

For this purpose, we shall show that 
\begin{equation}
\label{zeroint}
\int_B f(b,x) \, \phi(b) \, dm(b) = 0, \quad \textrm{for $\nu$-a.e. $x$ in $X$}
\end{equation}
for all $\phi \in L^1(B,m)$, and thus
\[
\int_X \int_B f(b,x) \phi(b) \psi(x) \, dm(b) \, d\nu(x) = 0
\]
for all $\phi \in L^1(B,m)$ and $\psi \in L^1(X,\nu)$, which shows that
$f$ must vanish identically, establishing ergodicity for the diagonal action
on $(B \times X,m \otimes \nu)$. To prove \eqref{zeroint}, we argue as 
follows. Define
\[
s(x) = \int_B f(b,x) \, dm(b)
\]
and note that
\begin{eqnarray*}
\int_G s(gx) \, d\mu(g) 
&=& 
\int_G \int_B f(b,gx) \, dm(b) \, d\mu(g) \\
&=&
\int_G \int_B f(g^{-1}b,x) \, dm(b) \, d\mu(g) \\
&=&
\int_B f(b,x) \,  dm(b) = s(x),
\end{eqnarray*}
since $m$ is $\mu$-harmonic. The following lemma now
shows that $s$ must vanish almost everywhere with respect 
to the measure $\nu$ (recall that $(X,\nu)$ is an ergodic $(G,\cmu)$-space).

\begin{lemma}
\label{harmonic constant}
Let $(X,\nu)$ be a $(G,\mu)$-space and suppose that 
$s \in L^\infty(X,\nu)$ satisfies the equation
\[
s = \int_G s(g^{-1} \cdot ) \, d\mu(g) \quad 
\textrm{in $L^\infty(X,\nu)$.}
\]
Then $f$ is essentially $G$-invariant. In particular, if $(X,\nu)$ is 
ergodic, then $f$ equals its $\nu$-integral almost everywhere.
\end{lemma}

\begin{proof}
We may assume that $s$ is real-valued. Since the support of $\mu$ is assumed to generate $G$, it suffices to show that
\[
\int_G \int_X \big| s(g^{-1}x) - s(x) \big|^2 \, d\mu^{*k}(g) \, d\nu(x) = 0
\]
for all $k$. However, upon expanding the square, and using the harmonicity of $\nu$, we see that
\[
\int_G \int_X \big| s(g^{-1}x) - s(x) \big|^2 \, d\mu^{*k}(g) \, d\nu(x)
= 2 \cdot \Big( 
\int_X |s(x)|^2 \, d\nu(x) 
- 
\int_X s(x) \, \Big( \int_G s(g^{-1}x) \, d\mu^{*k}(g) \Big) \, d\nu(x) \Big),
\]
which clearly vanishes by our assumption on $s$.
\end{proof}

Going back to the proof of Theorem \ref{K}, we can now conclude that 
\[
\int_B f(b,x) \, dm(b) = 0 \quad \textrm{for $\nu$-a.e. $x$ in $X$}, 
\]
and thus
\[
\int_B f(b,gx) \, dm(b) = \int_B f(g^{-1}b,x) \, dm(b) = \int_B f(b,x) , \sigma_m(g,b) \, dm(b) = 0,
\]
for almost every $x$ in $X$ and for all $g$ in $G$. In particular, we 
have
\[
\int_B f(b,x) \, \phi(b) \, dm(b) = 0, \quad \textrm{for $\nu$-a.e. $x$ in $X$},
\]
and for all $\phi$ in the linear span of all $\sigma_m(g,\cdot)$ as $g$ 
ranges over $G$. Hence Theorem \ref{K} will follow  
from the following simple lemma, which is essentially just a 
reformulation of Proposition \ref{furstenberg}.

\begin{lemma}
Let $(G,\mu)$ be a measured group and denote by $(B,m)$ its Poisson 
boundary. Then the linear span
\[
\cR_m = \linspan\Big\{ \frac{dg^{-1}m}{dm} \, : \, g \in G \Big\} \subset L^1(B,m)
\]
is dense in $L^1(B,m)$.
\end{lemma}

\begin{proof}
If this span would not be dense, then by Hahn-Banach's Theorem, there 
exists a non-zero functional $\phi \in L^1(B,m)^* = L^\infty(B,m)$ such
that 
\[
\int_X \phi(b) \, \frac{dg^{-1}m}{dm}(b) \, dm(b) = \int_X \phi(g^{-1}b) \, dm(b) = 0, \quad \forall \, g \in G,
\]
or equivalently, $P\phi(g) = 0$, where $P$ is as in Proposition \ref{furstenberg}. Since $P$ is an isomorphism, we conclude 
that $\phi$ vanishes identically, which is a contradiction. 
\end{proof}

We finish this section with yet another consequence of Proposition \ref{furstenberg} which seems to be rarely stressed in the literature. 
It was first observed by Kaimanovich in \cite{Ka92}, but the  
analogous case (in fact, concerning positive bi-harmonic functions) 
for amenable \emph{connected} measured group goes back to 
Raugi in \cite{Ra88}).

\begin{corollary}[Choquet-Deny, Blackwell, Kaimanovich]
\label{CD}
For every probability measure $\mu$ on a countable 
group $G$, there are no non-constant \emph{bounded} functions 
which are both left and right $\mu$-harmonic. In particular, 
measured \emph{abelian} groups do not admit any 
non-constant bounded harmonic functions.
\end{corollary}

\begin{remark}
The last assertion is immediate if $\mu$ is symmetric. If it is
not and $\phi$ is a bounded left $\mu$-harmonic function
with $\phi(e) = 0$ (and hence right $\cmu$-harmonic), then 
the symmetrized function
\[
\psi(g) = \phi(g) + \phi(g^{-1}), \quad g \in G,
\]
is a bounded left and right $\mu$-harmonic and thus identically
zero by the corollary above. This forces the identities 
$\phi(g) = - \phi(g^{-1})$ for all $g \in G$, and thus 
$\phi$ is both left and right $\mu$-harmonic, and hence
constant.
\end{remark}

\begin{proof}
Let $(B,m)$ denote the Poisson boundary of $(G,\mu)$ and suppose $f$ is 
a bounded left and right $\mu$-harmonic function on $G$. Let $\phi$ 
denote the unique element in $L^\infty(B,m)$ such that $f = P\phi$, 
where $P$ is the linear map defined in Proposition \ref{furstenberg}. Note
that the uniqueness of $\phi$ forces the identity
\[
\int_X \phi(g^{-1} \cdot ) \, d\mu(g) = \phi \quad \textrm{in $L^\infty(B,m)$}.
\]
By the last lemma, we conclude that $\phi$ is $G$-invariant and thus 
constant, since $(B,m)$ is ergodic. 
\end{proof}

\section{Weak mixing for $(G,\mu)$-spaces}

The aim of this section is to characterize weakly mixing $(G,\mu)$-spaces 
as exactly those which do not admit any non-trivial probability measure preserving factors with discrete spectrum. The author has not been able to 
locate an explicit formulation of this equivalence in the literature, although it should be stressed that
the characterization does follow from applying a series of classical and well-known techniques combined with the fact that WAP-actions are $\mu$-stiff in the sense of Furstenberg, which was established in \cite{FG10}. However, as the proof of 
this fact utilizes some serious machinery from the theory of Ellis semigroups and 
weakly almost
periodic functions, the route to the characterization of weakly mixing 
$(G,\mu)$-spaces (following these lines) is not very direct. We shall try to 
outline below a more direct approach. 

First recall that a non-singular $G$-space $(X,\nu)$ is \emph{weakly mixing}
if for every ergodic probability measure preserving $G$-space $(Y,\eta)$, the
diagonal $G$-action on $(X \times Y,\nu \otimes \eta)$ is ergodic. If $\nu$
is $G$-invariant, then this is equivalent to the absence of a non-trivial factor
with discrete spectrum, that is to say, a probability measure preserving 
$G$-space $(Z,\xi)$ with the property that the corresponding unitary (Koopman) representation on $L^2(Z,\xi)$ decomposes into a 
\emph{direct sum} of finite dimensional sub-representations. By a classical
theorem of Mackey in \cite{Ma}, an ergodic probability measure preserving $G$-space
with discrete spectrum is very special. Indeed, it
is always isomorphic to an isometric $G$-action on a compact homogeneous space,
that is to say, there exists a compact group $K$ and a closed 
subgroup $K_o < K$ and a homomorphism $\tau : G \ra K$
with dense image such that the $\tau(G)$-action on 
$K/K_o$ (with the Haar probability measure)
is isomorphic (as a $G$-space) to $(Z,\xi)$. In particular, if $G$
is a minimally almost periodic group (which means that there are
no non-trivial finite 
dimensional unitary $G$-representations whatsoever), then every
ergodic probability measure preserving $G$-space is automatically
weakly mixing. 

It is not true in general that an ergodic non-weakly mixing non-singular
$G$-space admits a probability measure preserving factor with discrete 
spectrum. In fact, Aaronson and Nadkarni constructs in \cite{AN87} a 
probability measure on a compact group, which is non-singular and 
ergodic with respect to dense cyclic subgroup, so that the corresponding
ergodic non-singular $\bZ$-space (which is certainly not weakly mixing) 
does not admit any non-trivial probability measure preserving factors whatsoever. 

However, in the category of $(G,\mu)$-spaces the situation is much nicer 
and the aim of this section is to give a self-contained proof of the following 
theorem, which is not new and certainly known to experts.

\begin{theorem}[Characterization of weak mixing]
\label{WM}
Let $(G,\mu)$ be a measured group and suppose $(X,\nu)$ is an 
ergodic $(G,\mu)$-space. Then $(X,\nu)$ is weakly mixing if and
only if $(X,\nu)$ does not admit a non-trivial probability measure preserving 
factor with discrete spectrum. In particular, if $G$ is minimally 
almost periodic, then every ergodic $(G,\mu)$-space is weakly
mixing.
\end{theorem}

\begin{remark}
In particular this theorem applies to the Poisson boundary of $(G,\mu)$,
which certainly does not have any probability measure preserving factors
whatsoever, and thereby giving yet another proof of the weak mixing of 
Poisson boundaries, originally due to Aaronson and Lema\'nczyk in \cite{AL05}. Note however that the theorem does \emph{not} directly 
apply to the setting of Theorem \ref{Kde} since products of Poisson 
boundaries are not $(G,\mu)$-spaces in general (unless of course, 
they are trivial). 
\end{remark}

Let us now begin the proof of Theorem \ref{WM}, which naturally falls into
two steps, both of which are essentially classical, and only the first step 
needs to be complemented with a less classical argument concerning 
probability measure preserving factors. As we have 
already mentioned above, this argument could be replaced by a nice, 
but not very elementary observation of Furstenberg and Glasner in 
\cite{FG10} about $\mu$-harmonic measures on WAP-spaces. However, 
since no self-contained proof of Theorem \ref{WM} seems to exist in the 
literature, it makes sense to outline a more direct route in this paper, and 
to collect here all the necessary arguments, although we do allow ourselves 
to be a bit sketchy in the more classical arguments. 

For the first step, we let $(G,\mu)$ be a measured group and $(X,\nu)$ is
an ergodic $(G,\mu)$-space. Suppose there exists an ergodic probability
measure preserving $G$-space $(Y,\eta)$ such that the diagonal $G$-action
on $(X \times Y,\nu \otimes \eta)$ is \emph{not} ergodic, that is to say, 
there exists a non-constant essentially bounded real-valued function 
$f$ on $X \times Y$. Without loss of generality, we can assume that $f$ is 
bounded by one, so that the map 
\[
\pi_f : X \ra B_1(L^2(Y,\eta))
\]
given by $\pi_f(x) = f(x,\cdot) \in B_1(L^2(Y,\eta))$ is well-defined for 
almost every $x$ in $X$, where $B_1(L^2(Y,\eta))$ denotes the unit 
ball in the Hilbert space $L^2(Y,\eta)$. Since $f$ is assumed to be 
$G$-invariant, one can readily verify that $\pi_f$ is a (weakly measurable) factor map from
$X$ into the $(G,\mu)$-space $(B_1(L^2(Y,\eta),\pi_*\nu)$, where $G$
acts on the unit ball $B_1(L^2(Y,\eta))$ via the (unitary) Koopman operator on 
$L^2(Y,\eta)$ (recall that $(Y,\eta)$ is measure-preserving). We note that 
since $f$ is non-constant, the corresponding factor is non-trivial. 

More generally, suppose $(\cH,\pi)$ is a unitary $G$-representation on a Hilbert space $\cH$. Then $\pi$ induces a (weakly continuous) action 
of $G$ on the unit ball $B_1(\cH)$, and if $\nu$ is a $\mu$-harmonic 
probability measure (with respect to this $G$-action) on $B_1(\cH)$, 
then we shall refer to $(B_1(\cH),\nu)$ as a \emph{Hilbertian $(G,\mu)$-space}. Theorem \ref{WM} will then follow from the following proposition. 

\begin{proposition}
Every Hilbertian $(G,\mu)$-space is measure-preserving and has 
discrete spectrum. 
\end{proposition}

We begin by proving that Hilbertian $(G, \mu)$-spaces are measure-preserving. To do so, we first note that by Stone-Weierstrass Theorem, 
the linear span of the constants and all functions of the form
\begin{equation}
\label{prods}
\phi(x) 
= 
\langle y_1, x\rangle \cdots \langle y_k, x\rangle , \quad x \in B_1(\cH),
\quad y_1, \ldots, y_k \in \cH
\end{equation}
is dense in $C(B_1(\cH))$, when $B_1(\cH)$ is equipped with the weak 
topology, and we wish to prove that
\begin{equation}
\label{harmeq}
\int_X \phi(g^{-1}x) \, d\nu(x) = \int_X \phi(x) \, d\nu(x), \quad \forall \, g \in G,
\end{equation}
for all $y_1,\ldots,y_k \in \cH$. Since
\begin{eqnarray*}
\int_G \int_{B_1(\cH)} \phi(g^{-1}x) \, d\nu(x) \, d\mu(g) 
&=&
\Big\langle y_1 \otimes \cdots \otimes y_k, 
\int_G \pi^{\otimes k}(g) \xi_\nu \, d\mu(g)
\Big\rangle, \\
&=&
\Big\langle y_1 \otimes \cdots \otimes y_k, 
\xi_\nu
\Big\rangle,
\end{eqnarray*}
for all $y_1,\ldots,y_k \in \cH$, where $\pi^{\otimes k}$ denotes the $k$-th
tensor product representation of $(\cH,\pi)$, and
\[
\xi_\nu = \int_X x \otimes \cdots \otimes x \, d\nu(x),
\]
we can conclude that 
\[
\pi^{\otimes k}(\mu) \xi_\nu
= 
\int_G \pi^{\otimes k}(g) \xi_\nu \, d\mu(g) = \xi_\nu.
\]
Hence \eqref{harmeq} will follow from the following simple lemma (applied to
all finite tensor product representations of $(\cH,\pi)$).
\begin{lemma}
Let $(\cH,\pi)$ be a unitary $G$-representation and suppose $\xi \in \cH$
satisfies $\pi(\mu)\xi = \xi$. Then $\xi$ is $\pi(G)$-invariant.
\end{lemma}

\begin{proof}
We may without loss of generality assume that $\|\xi\| = 1$. Since $\pi$ is
unitary, the equation $\pi(\mu)\xi = \xi$ simply means that a convex 
average of points of the form $\pi(g)\xi$, for $g$ in the support of $\mu$,
equals $\xi$. However, by the strict convexity of the unit ball in $\cH$, this
can only happen if $\pi(g)\xi = \xi$ for all $g$ in the support of $\mu$. 
Since the support of $\mu$ is assumed to generate $G$, we conclude that
$\xi$ is $\pi(G)$-invariant. 
\end{proof}

It remains to show that the measure-preserving $G$-space $(B_1(\cH),\nu)$
has discrete spectrum. For this purpose, we define the closed linear subspace
\[
\cH_o = \overline{\Big\{ v \in \cH \, : \, \textrm{the cyclic span of $v$ is finite-dimensional}
\Big\}} \subset \cH.
\]
One readily checks that $\cH_o$ is a sub-representation of $\cH$, and thus
its orthogonal complement $\cH_1$ is a sub-representation with the property
that it does not have any finite-dimensional sub-representations whatsoever.
Furthermore, we have
\[
B_1(\cH) = \Big\{ (\xi,\eta) \in \cH_o \oplus \cH_1 \, : \, \|\xi\|_0^2 + \|\eta\|_1^2 \leq 1 \Big\} \subset B_1(\cH_o) \times B_1(\cH_1).
\]
We have canonical continuous $G$-equivariant projections $\pi_o$ and $\pi_1$ from $B_1(\cH)$ onto $B_1(\cH_o)$ and $B_1(\cH_1)$ respectively,
and it is not hard to show that the $G$-space $(B_1(\cH_o),\pi_o)_*\nu)$
has discrete spectrum. Hence it suffices to show the following lemma. 
 
\begin{lemma}
Suppose $(\cH,\pi)$ is a unitary $G$-representation with no (non-trivial) finite-dimensional sub-representations. If $\nu$ is a $G$-invariant probability measure on $B_1(\cH)$, then it is concentrated at zero. 
\end{lemma}

\begin{proof}[Sketch of proof]
Since $C(B_1(\cH))$ is generated by limits of linear combinations of products
of the form as in \eqref{prods}, it suffices to show that
\begin{equation}
\label{suff}
\int_{B_1(\cH)} \big|\langle y,x\rangle\big|^2 \, d\nu(x) = 0, \quad \forall \, y \in \cH.
\end{equation}
In order to establish \eqref{suff}, we note that
\[
\int_{B_1(\cH)} \big|\langle y,x\rangle\big|^2 \, d\nu(x)
= 
\big\langle y \otimes y^*, \int_{B_1(\cH)} x \otimes x^* \, d\nu(x) \big\rangle, \quad \forall \, y \in \cH,
\]
where the $*$ refers to complex conjugation, and since $\nu$ is $G$-invariant, the vector
\[
\xi = \int_{B_1(\cH)} (x \otimes x^*) \, d\nu(x) \in \cH \otimes \cH^*
\]
is invariant under $\pi \otimes \pi^*(G)$. We wish to show that $\xi$ is
zero. To do so, we note that $\xi$ induces a \emph{compact and self-adjoint} linear map $K_\xi : \cH \ra \cH$ which is uniquely determined by
\[
\langle y,K_\xi z\rangle = \langle y \otimes z^*, \xi \rangle, \quad \forall \, y, z \in \cH.
\]
One readily checks that $K_\xi$ intertwines the representation $\pi$. By the spectral theorem for compact and self-adjoint linear maps, $\cH$ decomposes into a direct sum of the kernel of $K_\xi$ and \emph{finite}-dimensional eigenspaces for $K_\xi$. Since 
$\pi(g)$ commutes with $K_\xi$ for every $g$, each of these finite-dimensional 
subspaces must be invariant under $\pi$. However, since $\cH$ is assumed to 
completely lack finite-dimensional sub-representations, only the kernel of $K_\xi$
remains and we conclude that $K_\xi$ is trivial, i.e. $\xi$ is zero, which
finishes the proof. 
\end{proof}

\section{Biharmonic functions, coboundaries and central limit theorems}
In this section we shall discuss a novel perspective on a powerful
classical technique, often attributed to S.V. Nagaev \cite{Na1}, which
is designed to prove central limit theorems for certain classes of Markov chains. 
This technique is discussed at length in the book \cite{HeHe}, but in this paper we shall approach it in a slightly different way in the setting of 
random walks on groups.

We begin by describing a motivating example. Let $(G,\mu)$ be a
measured group and suppose $d$ is a left invariant distance function
on $G$ which satisfies the moment condition
\[
\int_G d(g,e)^{2+\eps} \, d\mu(g) < \infty, \quad \textrm{for some $\eps > 0.$}
\]
Let $(\Omega,\cP) = (G^{\bZ},\mu^{\bZ})$ and if $\omega$ is an 
element in $\Omega$, then we denote by $\omega_n$ the $n$'th
coordinate of $\omega$. One readily checks that $(\omega_n)$ is 
a sequence of independent $\mu$-distributed random variables on 
$G$, and we define $(z_n)$ to be the corresponding random walk, 
i.e. 
\[
z_n(\omega) = \omega_o \cdots \omega_{n-1}, \quad n \geq 1.
\] 
We note that the limit 
\[
\ell_d(\mu) = \lim_n \frac{1}{n} \int_\Omega d(z_n(\omega),e) \, d\bP(\omega) 
= 
\lim_n \frac{1}{n} \int_G d(g,e) \, d\mu^{*n}(g)
\]
exists and coincides with the drift of $(G,\mu,d)$ defined in the first
section of this paper. It follows from Theorem \ref{KL} that $\ell_d(\mu)$
is positive whenever $(G,\mu)$ is not Liouville, so in particular the drift
is positive if $G$ is non-amenable. In this case, the sequence 
\begin{equation}
\label{normseq}
Y_n = \frac{d(z_n,e) - n\ell_d(\mu)}{\sqrt{n}}, \quad  n \geq 1,
\end{equation}
of random variables fluctuates around zero, and it makes sense to ask
whether it has a non-trivial distributional limit. 

The aim of this section
is to outline a technique which isolates a class of triples $(G,\mu,d)$ for which the sequence $(Y_n)$ defined in \ref{normseq}
converges weakly to a non-degenerate Gaussian distribution on the real
line, that is to say, we wish to impose natural conditions on $G$, $\mu$
and $d$ such that for every continuous function $\varphi$ on $\bR$ with compact support, we have
\begin{equation}
\label{clt}
\lim_n 
\int_\Omega \varphi\Big(\frac{d(z_n,e) - n\ell_d(\mu)}{\sqrt{n}}\Big) \, d\bP
=
\frac{1}{\sigma \sqrt{2\pi}} \int_{-\infty}^\infty \varphi(t) e^{-t^2/2\sigma^2} \, dt,
\end{equation}
for some constant $\sigma > 0$. In probability theory, this convergence
is usually denoted by $Y_n \Rightarrow N(0,\sigma^2)$, and we shall adopt
this notation in this paper.

\subsection{Biharmonicity and central limit theorems}
Let us now briefly outline how the technique of S.V. Nagaev works in this setting. Its starting point is the fundamental observation (Proposition \ref{fund} below) that the values of \emph{biharmonic functions} along random walks always satisfy, under very weak assumptions, a central limit theorem. To make this
observation precise, we first recall that a real-valued function $\phi$ on $G$ is 
\emph{left Lipschitz} if the function
\[
\rho_\phi(g) = \sup_{s} \big|\phi(sg) - \phi(s)|, 
\]
is finite for every $g$ in $G$. One observes that if $\phi$ is left Lipschitz, 
then $\rho_\phi$ satisfies the triangle inequality
\[
\rho_\phi(g_1g_2) \leq \rho_{\phi}(g_1) + \rho_{\phi}(g_2), \quad \forall \, g_1, g_2 \in G,
\]
and thus, if $G$ is finitely generated,  it is bounded from above by any 
word metric on $G$. Furthermore, recall that a $\mu$-integrable function 
$\phi$ on $G$ is \emph{left $\mu$-quasiharmonic} if there exists a constant
$\ell(\phi)$ such that
\[
\int_G \phi(sg) \, d\cmu(s)  = \phi(g) + \ell(\phi), \quad \forall \, g \in G,
\]
and \emph{right $\mu$-quasiharmonic} if there exists a constant $r(\phi)$
such that
\[
\int_G \phi(gs) \, d\mu(s)  = \phi(g) + r(\phi), \quad \forall \, g \in G.
\]
Finally, we say that $\phi$ is \emph{bi-$\mu$-quasiharmonic} if it is left
and right $\mu$-quasiharmonic. By letting $g = e$ in the formulas above,
we see that if $\phi$ is bi-$\mu$-quasiharmonic, then $r(\phi) = \ell(\phi)$.\\

The fundamental observation upon which the technique of S.V. Nagaev hings can now be formulated as follows. 

\begin{proposition}
\label{fund}
Let $(G,\mu)$ be a symmetric measured group and suppose $\phi$ is a left Lipschitz, bi-$\mu$-quasiharmonic function on $G$ such that
\[
\int_G \rho_\phi(g)^{2+\eps} \, d\mu(g) < + \infty, \quad \textrm{for some $\eps > 0$.}
\]
If $\phi$ is not identically equal to $\ell(\phi)$, then there exists $\sigma > 0$ such that
\[
\frac{\phi(z_n) - n\ell(\phi)}{\sqrt{n}} \Rightarrow N(0,\sigma^2).
\]
\end{proposition}

\begin{proof}
Since $\phi$ is a right $\mu$-quasiharmonic function, the sequence 
\[
M_n = \phi(z_n) - n \ell(\phi), \quad n \geq 1,
\]
of measurable functions on $\Omega$ forms a martingale with respect to the filtration generated by the coordinates up to $n-1$. According to the martingale central limit theorem by McLeish in \cite{Mc74}, in order to prove the distributional convergence asserted in the proposition, it suffices to show that the sequence $(\psi_n)$ defined by
\[
\psi_n(\omega) 
= 
\frac{1}{\sqrt{n}} 
\max
\Big\{ 
\big| \phi(z_{j+1}(\omega)) - \phi(z_j(\omega)) - \ell(\phi)\big| \, : \, j=1,\ldots,n-1
\Big\}
\]
is uniformly integrable and $\int_\Omega \psi_n \, d\bP \ra 0$ as $n$ tends
to infinity, and
\begin{equation}
\label{convergence}
\lim_{n} 
\frac{1}{n} \sum_{j=1}^{n-1} 
\big| \phi(z_{j+1}(\omega)) - \phi(z_j(\omega)) - \ell(\phi)\big|^2 
= \sigma^2,
\end{equation}
almost everywhere with respect to $\bP$, where $\sigma$ is a positive 
constant. \\

Recall that by de la Vall\'ee-Poussin Theorem, $(\psi_n)$ is uniformly integrable if, but not only if, 
\begin{equation}
\label{goodbound}
\sup_n \int_\Omega \big|\psi_n\big|^{2+\eps} \, d\bP < \infty,
\end{equation}
and thus to prove the two first assertions it suffices to show that
\[
\int_\Omega |\psi_n|^{2+\eps} \, d\bP \leq \frac{1}{n^{\eps/2}} \int_{G} \rho_\phi(g)^{2+\eps} \, d\mu(g),
\]
since the last integral is finite by assumption. 

First note that the 
shift map $\tau : \Omega \ra \Omega$ given by $\tau(\omega)_n = \omega_{n+1}$ preserves the probability measure $\bP$ on $\Omega$ 
and is ergodic. Secondly, we have
\[
\psi_n(\omega) 
\leq 
\frac{1}{\sqrt{n}} \cdot 
\max\Big\{ \rho_\phi(\omega_j) \, : \, 1 \leq j \leq n-1 \Big\}
\]
for all $n$, so that if we define $v(\omega) = \rho_\phi(\omega_o)$,
then $v \in L^{2+\eps}(\Omega,\bP)$, and it is a straightforward exercise
to show that
\[
\int_\Omega \max_{1 \leq j \leq n-1} |v(\tau^j \omega)|^{2+\eps} \, d\bP \leq 
\frac{1}{n^{\eps/2}} \int_{\Omega} |v(\omega)|^{2+\eps} \, d\bP(\omega) 
=
\frac{1}{n^{\eps/2}} \int_{G} \rho_\phi(g)^{2+\eps} \, d\mu(g) \ra 0,
\]
Hence it remains to show the convergence in \eqref{convergence}. For
this purpose, we define the sequence 
\[
u_j(\omega) = 
\big| \phi(\omega_{-j} \cdots \omega_o) - \phi(\omega_{-j}\cdots \omega_{-1}) - \ell(\phi)\big|^2,
\]
so that we can write 
\[
u_j(\tau^{j}\omega) = \big| \phi(z_{j+1}(\omega)) - \phi(z_j(\omega)) - \ell(\phi)\big|^2, \quad \forall \, j \geq 1.
\]
We wish to prove that there exists a positive constant $\sigma > 0$ such
that 
\[
\sigma^2 = \lim_n \frac{1}{n} \sum_{j=1}^n u_j(\tau^j\omega) 
\]
almost everywhere. By Breiman's Lemma (see e.g. Lemma 14.34 in \cite{Gla}), it suffices to show that 
\[
\int_\Omega \sup_j u_j \, d\bP < \infty
\]
and that there exists a function $u \in L^1(\Omega,\bP)$ such that 
$u_j \ra u$ almost surely and in the $L^1$-norm. Indeed, if this
is the case, then
\[
\sigma^2 = \lim_n \frac{1}{n} \sum_{j=1}^n u_j(\tau^j\omega) 
= \int_\Omega u \, d\bP,
\]
almost surely and $\sigma = 0$ if and only if $u$ vanishes almost everywhere. To prove the existence of a function $u$ as above, we define the sequence 
\[
N_j = \phi(\omega_{-j} \cdots \omega_o) - \phi(\omega_{-j}\cdots \omega_{-1}) - \ell(\phi),
\]
so that $u_j = |N_j|^2$, and since $\phi$ is \emph{left} 
$\mu$-quasiharmonic, we see that $(N_j)$ is a martingale
with respect to the filtration generated by the coordinates 
from $-j$ to $0$. Furthermore, since $\phi$ is left Lipschitz,
we also have that
\[
C = \sup_j \int_\Omega |N_j|^{2} \, d\bP \leq \int_G \rho_\phi(g)^{2} \, d\mu(g) + \ell(\phi)^2 + 2 \cdot \ell(\phi) \cdot \int_G \rho_\phi(g) \, d\mu(g),
\]
and thus $(N_j)$ is a $L^2$-bounded martingale. In particular, by 
the classical Martingale Convergence Theorem, there exists a function
$N_\infty$ in $L^2(\Omega,\bP)$ such that $N_j \ra N_\infty$ almost
everywhere and 
\[
\lim_j \int_\Omega \big|N_j - N_\infty|^2 \, d\bP = 0
\]
and thus, with $u = |N_\infty|^2$, we have $u_j \ra u$ almost everywhere and
\[
\int_\Omega \big|u_j - u\big| \, d\bP \leq 
\int_\Omega \big|(N_j-N_\infty)(N_j + N_\infty)\big| \, d\bP
\leq 4 \cdot C \cdot \int_\Omega \big|N_j - N_\infty|^2 \, d\bP \ra 0.
\]
Finally, we need to show that $u$ does not vanish almost 
everywhere with respect to $\bP$. Note that if $u$
vanishes almost everywhere, then so does $N_\infty$ and 
thus
\[
\lim_j \phi(\omega_{-j} \cdots \omega_o) - \phi(\omega_{-j}\cdots \omega_{-1}) = \ell(\phi)
\]
almost everywhere. Hence, for any fixed $j_o$, by calculating the 
conditional expectation of the limit with respect to the $\sigma$-algebra 
generated by all coordinates strictly below $-j_o$, we conclude that
\[
\phi(\omega_{-j_o} \cdots \omega_{o}) - \phi(\omega_{-j_o} \cdots \omega_{-1}) = \ell(\phi),
\]
almost everywhere. In particular, since $\mu$ is assumed to generate
$G$ as a semigroup, we have $\phi = \ell(\phi)$ everywhere, which we
have assumed is not the case.
\end{proof}

\subsection{Constructing bi-quasiharmonic functions}
We now return to our motivating example. As we have seen
in Subsection \ref{subsecLD}, given any triple $(G,\mu,d)$,
there exists a sequence $(n_j)$ such that the limit
\[
\phi(g) = \lim_{j \ra \infty} \frac{1}{n_j}
\sum_{k=0}^{n_j-1} \int_G \big(d(g,x)-d(x,e)\big) \, d\mu^{*k}(x)
\]
exists for all $g$ in $G$, and the function $\phi$ satisfies 
\[
\int_G \phi(sg) \, d\cmu(s) = \phi(g) + \ell_d(\mu)
\]
and 
\[
\phi(g) \leq d(g,e) \qand \rho_\phi(g) = \sup_{s} \big|\phi(sg) - \phi(s)\big| \leq d(g,e), \quad \forall \, g \in G.
\]
In particular, $\phi$ is left Lipschitz and left $\mu$-quasiharmonic. 
Furthermore, if we write
\[
\frac{d(z_n,e) - n\ell_d(\mu)}{\sqrt{n}} 
= 
\frac{d(z_n,e) - \phi(z_n)}{\sqrt{n}} 
+ 
\frac{\phi(z_n) - n\ell_d(\mu)}{\sqrt{n}},
\]
then the first term is non-negative and converges to zero in the $L^1$-norm
if and only if
\begin{equation}
\label{neglect}
\lim_n \frac{1}{\sqrt{n}} \Big( \int_G d(g,e) \, d\mu^{*n}(g) - n \ell_d(\mu) \Big) = 0.
\end{equation}
Hence, under condition \ref{neglect}, the question whether \ref{clt} holds is
completely reduced to the question whether 
\begin{equation}
\label{cltqh}
\frac{\phi(z_n) - n\ell_d(\mu)}{\sqrt{n}} \Rightarrow N(0,\sigma^2)
\end{equation}
for some positive constant $\sigma$. \\

Unfortunately, there is no reason in general to expect that $\phi$ is also 
right $\mu$-quasiharmonic so that Proposition \ref{fund} can be directly applied. We approach this serious problem as follows. Let $(B,m)$ be the
Poisson boundary of $(G,\mu)$ and note that for every $u \in L^\infty(B)$,
the function 
\[
\phi_u(g) = \phi(g) + \int_G u(g^{-1}b) \, dm(b), \quad g \in G,
\]
is again left Lipschitz and left $\mu$-quasiharmonic. Furthermore, \eqref{cltqh} holds for $\phi_u$ if and only if it holds for $\phi$. 
Hence it makes sense to ask whether we can find $u \in L^\infty(B,m)$ 
such that $\phi_u$ is right $\mu$-quasiharmonic. It turns out that there
is a simple criterion for this. Indeed, since $\phi$ is left Lipschitz, one 
can readily check that the function
\[
\widehat{\psi}(s) = \int_G \big( \phi(sg) - \phi(s) \big) \, d\mu(g),
\]
is bounded and left $\mu$-harmonic, and thus it corresponds via 
the Poisson transform (discussed in the first section of this paper)
to an element $\psi$ in $L^\infty(B)$ (which we shall refer to as the
\emph{right $\mu$-obstruction}), with the property that
\[
\int_B \psi(b) \, dm(b) = \ell_d(\mu).
\]
We observe that $\phi_u$ is right $\mu$-quasiharmonic if and only if $u$ satisfies the "cohomological equation"
\begin{equation}
\label{cohomo}
u(b) - \int_B u(s^{-1}b) \, d\mu(s) = \psi(b) - \ell_d(\mu), \quad 
\textrm{a.e. $[m]$.}
\end{equation}
For many triples $(G,\mu,d)$ of interest, such as Gromov hyperbolic groups equipped with symmetric probability measures with finite exponential moments, one can show that $\psi -\ell_d(\mu)$ must belong to a certain subspace $\cB \subset L^\infty(B,m)$ consisting of "smooth" functions with
zero $m$-integrals, which admits a seminorm $N_o$ with the property that
\[
N(u) = \|u\|_\infty + N_o(u)
\]
is a norm on $\cB$ and there exist $0 < \tau < 1$ and an integer 
$n_o$ such that the convolution operator
\[
Q_\mu u(b) = \int_G u(g^{-1} b) \, d\mu(g)
\]
satisfies the contraction bound
\begin{equation}
\label{boundN}
N_o(Q_\mu^{n_o}u) \leq \tau \cdot N_o(u)  \quad \forall \, u \in \cB.
\end{equation}
Note that once such a bound has been established, it is not hard to show that
the von Neumann series
\[
u = \sum_{n \geq 0} Q_{\mu}^{*n}\big(\psi-\ell_d(\mu)\big)
\]
is a well-defined element in $\cB$ which solves the equation \ref{cohomo}. The
main aim of the rest of this section will be to single out a class of symmetric measured 
groups which comes equipped with a "natural" weakly dense semi-normed subspace of $L^\infty(B,m)$, which one should think of as "measurably H\"older continuous" functions, on which $Q_\mu$ satisfies the above contraction bound.

\subsection{Besov spaces defined by product currents}
Let $(G,\mu)$ be a countable symmetric measured group and suppose 
$(X,\nu)$ is a compact $(G,\mu)$-space, that is to say, $X$ is a compact metrizable space equipped with an action of $G$ by homeomorphisms 
such that $\nu$ satisfies the equation
\[
\int_G \int_X \phi(s^{-1}x) \, d\nu(s) \, d\mu(s) = \int_X \phi(x) \, d\nu(x)
\]
for all $\phi \in C(X)$. If $\nu$ is non-atomic, then we can think of 
the product measure $\nu \otimes \nu$ as a probability measure on
the (in general) \emph{non-compact} space $\partial^2 X = X \times X \setminus \Delta X$, where $\Delta X$ denotes the (closed) diagonal subspace in $X \times X$. A non-negative Borel measurable function 
$\rho$ on $\partial^2 X$ is called a \emph{product current} if the 
(possibly infinite) Borel measure $\eta$ on $\partial^2 X$ defined by
\[
\int_{\partial^2 X} \phi(x,y) \, d\eta(x,y)
= 
\int_{\partial^2 X} \phi(x,y) \, \rho(x,y) \, d\nu(x) \, d\nu(y)
\]
is invariant with respect to the diagonal action of $G$ on 
$\partial^2 X \subset X \times X$. One can readily check
that this condition simply translates to the validity of the 
equation
\begin{equation}
\label{equivariance}
\rho(gx,gy) \, \sigma_{\nu}(g,x) \, \sigma_{\nu}(g,y) = \rho(x,y)
\end{equation}
for all $g$ in $G$ and for almost every $(x,y)$ with respect to 
the product measure $\nu \otimes \nu$. 

We stress that not 
every $(G,\mu)$-space admits a product current. However, 
certain classes of countable groups, such as Gromov hyperbolic 
groups and lattices in higher rank Lie groups, carry symmetric
probability measures with the property that their Poisson boundaries
(in some compact model) admit "natural" and "geometrically defined" 
product currents. We refer the reader to Section 5 of the paper \cite{BHM} 
for a detailed discussion about product currents 
for Gromov hyperbolic measured groups. In this case, $X$ is the Gromov
boundary of the hyperbolic group $G$, equipped with a certain distance function $d_o$, and $\rho$ is roughly proportional to $d_o(x,y)^{-D}$, 
where $D$ is a constant related to the Hausdorff dimension of $X$.

Given a product current $\rho$ for a $(G,\mu)$-space $(X,\nu)$ and given 
$\eps > 0$, we define a semi-norm $N_{\rho,\eps}$ on a subspace 
$\cB_{\rho,\eps} \subset L^\infty(X,\nu)$ by
\[
N_{\rho,\eps}(u) = 
\int_{\partial^2 X} \big|u(x) - u(y)\big| \, 
\rho(x,y)^{\frac{1}{2} + \eps} \, d\nu(x) \, d\nu(y),
\]
where $\cB_{\rho,\eps}$ consists of those elements in $L^\infty(X,\nu)$
with finite $N_{\rho,\eps}$-seminorms. We shall refer to linear space $(\cB_{\rho,\eps},N_{\rho,\eps})$ as the \emph{Besov space associated to $\rho$ of order 
$\eps$}, and since $\rho$ usually blows up close to the diagonal, we may think of $\cB_{\rho,\eps}$ as 
a "measurable" replacement of H\"older continous functions on $X$. As 
the following proposition will show, there is a simple criterion for the 
validity of the contraction bound \ref{boundN} for $Q_\mu$ acting on the space $\cB_{\rho,\eps}$.

\begin{proposition}
Let $(G,\mu)$ be a countable measured group and suppose $(X,\nu)$ is a
$(G,\mu)$-space which admits a product current $\rho$. Given $\eps > 0$
and an integer $n$, we define
\[
\tau_{\eps,n} = \esssup \int_G \sigma_\nu(g,\cdot)^{1-2\eps} \, d\mu^{*n}(g).
\]
Then $N_{\rho,\eps}(Q_\mu^{n} u) \leq \tau_{\eps,n} \cdot N_{\rho,\eps}(u)$ for all $u \in \cB_{\rho,\eps}$.
\end{proposition}

\begin{proof}
First recall that
\[
\rho(sx,sy) \, \sigma_\nu(s,x) \, \sigma_\nu(s,y) = \rho(x,y)
\]
for almost every $(x,y)$ with respect to $\nu \otimes \nu$. Hence, 
we have
\begin{eqnarray*}
N_{\rho,\eps}(Q_{\mu}^{*n}u)
&\leq &
\int_G \int_{\partial^2 X}
\big|u(s^{-1}x) - u(s^{-1}y)\big| \, \rho(x,y)^{\frac{1}{2}+\eps} \, 
d\nu(x) \, d\nu(y) \, d\mu^{*n}(s) \\
&= &
\int_G \int_{\partial^2 X}
\big|u(x) - u(y)\big| \, \rho(sx,sy)^{\frac{1}{2}+\eps} \, \sigma_\nu(s,x) \, \sigma_\nu(s,y) \, 
d\nu(x) \, d\nu(y) \, d\mu^{*n}(s) \\
&= &
\int_G \int_{\partial^2 X}
\big|u(x) - u(y)\big| \, \rho(x,y)^{\frac{1}{2}+\eps} \, 
\sigma_\nu(s,x)^{\frac{1}{2}-\eps} \, \sigma_\nu(s,y)^{\frac{1}{2}-\eps} \, 
d\nu(x) \, d\nu(y) \, d\mu^{*n}(s) \\
&\leq &
\Big( \esssup \int_G \sigma_\nu(s,\cdot)^{\frac{1}{2}-2\eps} \, d\mu^{*n}(s) \Big) \cdot N_{\rho,\eps}(u),
\end{eqnarray*}
where we in the last line used H\"older's inequality twice.
\end{proof}

\subsection{Minimality and non-invariance force contraction}
The aim of the final subsection of this section will be to isolate natural 
conditions on a compact $(G,\mu)$-space $(X,\nu)$ which will force the existence
of an integer $n$, for every given $\eps > 0$, such that
\begin{equation}
\label{cocyclebnd}
\tau_{n,\eps} = \esssup \int_G \sigma_{\nu}(s,\cdot)^{1-2\eps} \, d\mu^{*n}(s) < 1.
\end{equation}
We shall henceforth assume that the functions $x \mapsto \sigma_\nu(s,x)$
are continuous for every $s$ in $G$. Although this assumption is not absolutely necessary, it will simplify many of the arguments below. Furthermore, we may without loss of generality assume that the identity 
belongs to the support of $\mu$. Indeed, if not, then we can replace $\mu$
with the probability measure 
\[
\mu_o = \frac{1}{2}\delta_e + \frac{1}{2} \mu,
\]
with respect to which $\nu$ is still stationary, and \eqref{cocyclebnd} holds
for $\mu$ if and only if it holds for $\mu_o$. Note that the supports of 
$\mu_o^{*n}$ forms an \emph{increasing} family of sets in $G$ which 
asymptotically exhausts $G$.\\

First recall that by by \eqref{cocycleharm} (which holds for every $(G,\mu)$-space), we have
\[
\int_G \sigma_\nu(s,x) \, d\mu^{*n}(s) = 1
\]
for all $n$ and for almost every $x$ in $X$,  In particular, $\tau_{n,\eps}$ is
always bounded by one for all $n$ and $\eps$, and \eqref{cocyclebnd} fails 
if and only if for every $n$, there exists $x_n \in X$ such that 
\[
\int_G \sigma_\nu(s,x_n)^{1-2\eps} \, d\mu^{*n}(s) = 1.
\]
In other words, for every $n$, we have equality in H\"older's inequality (when 
integrating against $\mu^{*n}$), which clearly forces the identities
\[
\sigma_\nu(s,x_n) = 1 \quad \forall \, s \in \supp \mu^{*n}
\]
for all $n$. Let $x_\infty$ be an accumulation point of the sequence $(x_n)$
in $X$. Since $\sigma_\nu(s,\cdot)$ is continuous for every $s$ and the supports of $\mu^{*n}$ is an increasing exhausting family of sets in $G$, we conclude that 
\[
\sigma_\nu(s,x_\infty) = 1, \quad \forall \, s \in G.
\]
Let us now further assume that the $G$-action on $X$ is \emph{minimal}, i.e. 
every $G$-orbit is dense. Then, by the cocycle equation \eqref{multcocycle}, which holds for every $(G,\mu)$-space, we have 
$\sigma_\nu(s,tx_\infty) = 1$ for all $s,t$ in $G$, and since $Gx_\infty$
is dense and $\sigma_\nu(s,\cdot)$ is continuous, we conclude that 
$\sigma_\nu(s,x) = 1$ for all $s$ in $G$ and $x$ in $X$, or equivalently, 
$\nu$ is $G$-invariant. We summarize the above discussion in the following 
proposition.

\begin{proposition}
\label{cont}
Let $(G,\mu)$ be a countable measured group and suppose $(X,\nu)$ is 
a compact minimal $(G,\mu)$-space such that $\sigma_\nu(s,\cdot)$ is continuous for every $s$ in $G$. If $\nu$ is not $G$-invariant, then for
every $\eps > 0$, there exists an integer $n$ such that
\[
\sup \int_G \sigma_\nu(s,\cdot)^{1-2\eps} \, d\mu^{*n}(s) < 1.
\]
\end{proposition}

\begin{remark}
The assumptions in the last proposition are satisfied for every symmetric
probability measure $\mu$ with finite exponential moments (with respect 
to the any word metric) on any non-elementary Gromov hyperbolic group, where $(X,\nu)$ denotes its Gromov boundary and $\nu$ is the unique 
$\mu$-stationary measure on $X$. Hence Proposition \ref{cont} gives a 
new proof of the main technical estimate in the author's paper \cite{Bj08}.

Although Proposition \ref{cont} assumes a lot about the topological 
and dynamical structure of $(X,\nu)$, there is no assumption about 
the moments of $\mu$. In particular, Proposition \ref{cont}, as well as 
the discussions about product currents proceeding it, also apply to the 
Furstenberg boundary action of a lattice $G$ in a simple Lie group $H$, 
at least when the $\mu$-stationary measure $\nu$ on the Furstenberg boundary $H/P$ (here $P$ is a minimal parabolic subgroup of $H$) 
belong to the Haar measure class. Such probability measures on the lattice always exist (see e.g. \cite{Fu}), but they tend to have very heavy tails. Since
the Furstenberg boundary of a simple Lie group, equipped with the Haar 
measure, always admits a product current (upon identifying a conull subset 
of $H/P \times H/P$ with $H/A$, where $A$ is the (unimodular) split torus of $G$), Proposition \ref{cont} in particular
implies that every function in the associated Besov space $\cB_{\rho,\eps}$ 
is in fact of the form $\phi - \mu * \phi$ for some 
$\phi \in \cB_{\rho,\eps}$.
\end{remark}
\section{Product sets in groups}

This final section is concerned with the structure of 
difference sets in free groups, and the aim here is to give a short and
rather elementary proof of a weaker version of a recent theorem by 
the author and A. Fish (Theorem 1.1 in \cite{BF3}). We begin by providing some background and motivation. \\

A significant part of additive combinatorics is concerned with special 
instances of the following phenomenology: If $G$ is a countable group and 
$A, B \subset G$ are "large" subsets, then the product set $AB$ should
exhibit "substantial sub-structures". The exact meanings of these 
notions varies a lot depending on the context, 
and in this section we shall only be concerned with (partially) extending the 
following result by Khintchine \cite{Kh35} and F\o lner \cite{Fo}, which was
one of the first observations of this phenomenology (at least in the setting 
of discrete groups).

\begin{theorem}
\label{FO}
Suppose $A_1, \ldots, A_k \subset \bZ$ are subsets which are "large" in
the sense that
\[
\varlimsup_{n \ra \infty} \frac{|A_i \cap [-n,n]|}{2n+1} > 0, \quad \forall \, i=1,\ldots,k. 
\]
Then their difference sets contain "substantial sub-structures", in the sense
that there exists a finite set $F \subset \bZ$ such that
\[
F + \bigcap_{i=1}^k (A_i-A_i) = \bZ.
\]
\end{theorem}

The additive group of integers is of course nothing but the free group on one generator. A first naive attempt to extend Theorem \ref{FO} to free group on two or more generators could be devised along the following lines. Let $\bF_2$ denote the free group on two (free) generators $a$ and $b$ and let $B_n$ denote the ball of radius $n$ with respect to these generators, that is to say, $B_n$ consists of all the words in $a$ and $b$ and their inverses whose reduced form have length at most $n$. In analogy with Theorem \ref{FO} (where the "balls" with respect to the one free generator $1$ are simply given by the interval $[-n,n]$), we say that a set 
$A \subset \bF_2$ is \emph{upper large} if 
\[
\varlimsup_{n \ra \infty} \frac{|A \cap B_n|}{|B_n|} > 0.
\]
However, we warn the reader that upper large sets could be \emph{very} sparse
in $\bF_2$; for instance, given any increasing sequence $(r_i)$ of positive
integers, the set
\begin{equation}
\label{sparse}
A = \bigcup_{i=1}^\infty \big(B_{r_{i}+1} \setminus B_{r_i}\big) \subset \bF_2
\end{equation}
satisfies 
\[
\varlimsup_{n \ra \infty} \frac{|A \cap B_n|}{|B_n|} \geq \frac{2}{3}.
\]
We do not expect to say anything intelligent about differences of sets like 
these, so we slightly modify our notion of largeness to exclude too sparse
examples. Define the \emph{sphere} $S_n$ of radius $n$ by $S_n = B_n \setminus B_{n-1}$ and say that a set $A \subset G$ is \emph{large} if 
\[
\varlimsup_{m \ra \infty} \frac{1}{m} \sum_{n=1}^m \frac{|A \cap S_n|}{|S_n|} > 0.
\]
We see that for a set as in \eqref{sparse} to be large in this sense, serious 
growth constraints on the sequence $(r_i)$ have to be imposed, so the notion
of largeness is \emph{strictly} weaker than upper largeness. 

One can now ask whether something like Theorem \ref{FO} could be true for 
large sets. However, already simple considerations show that great care has to
be taken to even formulate the right statement. Indeed, it is not hard to 
construct (and we refer to \cite{BF3} for details) large 
subsets $A_1,A_2,A_3 \subset \bF_2$ such that
\[
A_1 A_1^{-1} \cap A_2 A_2^{-1} \cap A_3 A_3^{-1} = \{0\} 
\]
and for which there is no finite subset $F \subset G$ such that $FA_iA_i^{-1} = \bF_2$ 
for \emph{some} $i = 1,2,3$. However, the situation is not completely hopeless if one is willing to slightly weaken the notion of "substantial sub-structure"
as the following recent observation (see Corollary 1.2 in \cite{BF3}) by the author and A. Fish shows. 

\begin{theorem}[Bj\"orklund-Fish, weak version]
\label{BF}
Suppose $A \subset \bF_2$ is "large" in the sense that
\[
\varlimsup_{m \ra \infty} \frac{1}{m} \sum_{n=1}^m \frac{|A \cap S_n|}{|S_n|} > 0.
\]
Then there exists a finite set $F \subset \bF_2$ such that 
$FAA^{-1}$ contains a right translate of every finite subset 
of $G$. 
\end{theorem}

We stress that this is not the formulation of Corollary 1.2. in \cite{BF3}, so 
we first take a moment to rewrite Theorem \ref{BF} in a language which  
better align with the present paper (and with \cite{BF3}). Let $G = \bF_2$ and
define the probability measures $(\sigma_n)$ on $G$ (uniform sphere averages) by 
\[
\sigma_o = \delta_e \qand \sigma_n = \frac{1}{|S_n|} \sum_{s \in S_n} \delta_s, \quad \textrm{for $n \geq 1$}.
\]
It is not hard to verify the relations
\begin{equation}
\label{rec}
\sigma_n * \sigma_1 = \frac{3}{4} \cdot \sigma_{n+1} + 
\frac{1}{4} \cdot \sigma_{n-1}, \quad \forall \, n \geq 1,
\end{equation}
which in particular shows that every $\sigma_n$ can be written as
a convex combination of convolution powers of $\sigma_1$ and 
$\delta_e$.

Let $M(G)$ denote the convex set of all means on $G$, i.e. the set 
of all linear functionals on $\ell^\infty(G)$ which are positive and 
unital (i.e. $\lambda(1) = 1)$. We note that every mean $\lambda$ 
gives rise to a \emph{finitely additive} probability measure $\lambda'$
on $G$ via the formula
\[
\lambda'(C) = \lambda(\chi_C), \quad C \subset G,
\]
and by the Banach-Alaoglo's Theorem, the set $M(G)$ is compact with respect to the
weak*-topology. In particular, every sequence $(\lambda_i)$ of the form
\[
\lambda_i = \frac{1}{m_i} \sum_{n=1}^{m_i} \sigma_n, \quad i \geq 1,
\]
for some increasing sequence $(m_i)$, must have at least one cluster point $\lambda$, which by the relations in \eqref{rec} is necessarily left $\sigma_1$-harmonic (note that $\sigma_1$ is symmetric), that is to say
\[
\int_G g \cdot \lambda(\varphi) \, d\sigma_1(g) = \lambda(\varphi), \quad \forall \, \varphi \in \ell^\infty(G),
\]
where $G$ acts on $\ell^\infty(G)$ (and hence on its dual via the transpose map) by the
left regular representation. In particular, if we choose a sequence $(m_i)$ such that
\[
\lim_{i \ra \infty} \frac{1}{m_i} \sum_{n=1}^{m_i} \frac{|A \cap S_n|}{|S_n|} > 0,
\]
and a cluster point $\lambda$ of the corresponding sequence of means as above, then
$\lambda'(A) > 0$. The aim is now to show that this condition automatically forces the 
existence of a finite set $F \subset G$ such that $FAA^{-1}$ contains a right translate of
of every finite subset of $G$.

It will be convenient to adopt a slightly more general perspective on these matters. Let 
$(G,\mu)$ be a countable symmetric measured group. We say that an element $\lambda \in \cM(G)$ is \emph{left $\mu$-harmonic} if 
\[
\int_G \lambda(\varphi(g^{-1} \cdot)) \, d\mu(g) = \lambda(\varphi), \quad \forall \, \varphi \in \ell^\infty(G).
\]
We say that a set $T \subset G$ is \emph{right thick} if it contains a right translate of every
finite subset of $G$, that is to say, if for every finite subset $F \subset G$, there exists 
$g \in G$ such that $Fg \subset T$. It is not hard to see that a set $T \subset G$ is
right thick if and only if for every finite set $F \subset$, the intersection of all left
translates of the form $fT$, with $f \in F$, is non-empty. In particular, if $\lambda$ is 
a left $\mu$-harmonic mean on $G$ such that $\lambda'(T) = 1$, then 
\[
\int_G \lambda'(gT) \, d\mu^{*k}(g) = \lambda'(T) = 1, \quad \forall \, k \geq 1,
\] 
which shows that $\lambda'(gT) = 1$ for all $g \in G$. Since $\lambda'$ is a finitely 
additive measure, we conclude that for every finite set $F \subset G$, the intersection
of all left translates $fT$, with $f \in F$, still has full $\lambda'$-measure (so in particular
it is non-empty), which shows that $T$ must be right thick. \\

Theorem \ref{BF} will now follow from the following proposition.

\begin{proposition}
\label{mainBF}
Let $(G,\mu)$ be a measured group and suppose $A \subset G$ has 
positive measure with respect to some left $\mu$-harmonic mean on $G$.
Then there exists a finite set $F \subset A$ such that $FAA^{-1}$ has 
measure one with respect to some left $\mu$-harmonic
mean on $G$.
\end{proposition}

To prove this proposition, we will need the following result, which is not hard, and follows
from quite standard correspondence principles. However, the author is not aware of a (short) proof which avoids various technical manipulations with extreme points in the simplex of $\mu$-harmonic measures on compact $G$-spaces. We shall 
therefore omit the proof, and refer the interested reader to Proposition 1.2 in \cite{BF3},
where a much stronger result is proven.

\begin{lemma}
Fix $\eps > 0$ and suppose $A \subset G$ has positive measure with respect to some 
left $\mu$-harmonic mean. Then there exists a finite set $F \subset G$ and a (possibly
different) left $\mu$-harmonic mean $\eta$ on $G$ such that $\eta(FA) \geq 1 - \eps$. 
\end{lemma}

If one is willing to take this lemma for granted, then we argue as follows. Suppose $A \subset G$ has positive $\lambda'$-measure for some left $\mu$-harmonic mean $\lambda$ on $G$. Fix $\eps > 0$ and find, by the previous lemma, a finite set 
$F \subset G$ and a left $\mu$-harmonic mean $\eta$ on $G$ such that 
\[
\eta(FA) \geq 1 - \eps \cdot \eta(A).
\]
We note that
\[
FAA^{-1} \supset \big\{ g \in G \, : \, \eta(FA \cap gA) > 0 \big\} \supset 
\big\{ g \in G \, : \, \eta(gA) > \eps \cdot \eta(A) \Big\},
\]
and the function
\[
u(g) = \eta(gA) - \eps \cdot \eta(A), \quad g \in G,
\]
is a real-valued bounded left $\mu$-harmonic function on $G$. Furthermore, if 
$0 < \eps < 1$, then $u(e) > 0$ and $u$ is \emph{positively correlated} in the
sense that
\[
\|u\|_\infty = \sup\big\{ u(g) \, : \, g \in G \big\},
\]
so Proposition \ref{mainBF} will follow from the "zero-one law" stated below. 

\begin{lemma}
\label{measone}
Let $(G,\mu)$ be a measured group and suppose $u$ is a bounded 
real-valued left $\mu$-harmonic function on $G$. Define the set 
\[
S_u = \Big\{ g \in G \, : \, u(g) > 0 \Big\} \subset G.
\]
If $u$ is positively correlated, then there exists a left $\mu$-harmonic mean 
which gives measure one to the set $S_u$.
\end{lemma}

To prove this lemma, we first note that for any mean $\lambda$ on $G$, for any
$\eps > 0$ and for every bounded function $u$ on $G$, we have
\begin{eqnarray*}
\lambda'\big(\big\{ g \in G \, : \, u(g) > 0 \big\} 
&\geq&
\lambda'\big(\big\{ g \in G \, : \, u(g) \geq (1-\eps) \cdot \|u\|_\infty \big\}  \\
&=& 
1 - \lambda'\big(\big\{ g \in G \, : \, u(g) < (1-\eps) \cdot \|u\|_\infty \big\} \\
&=&
1 - \lambda'\big(\big\{ g \in G \, : \, \|u\|_\infty - u(g) > \eps \cdot \|u\|_\infty \big\} \\
&\geq&
1 - \frac{1}{\eps \cdot \|u\|_\infty} \cdot (\|u\|_\infty - \lambda(u)),
\end{eqnarray*}
by Chebyshev's inequality (which works equally well for finitely additive probability 
measures). 
Hence it suffices to show that whenever $u$ is a positively correlated left $\mu$-harmonic function, there exists a left $\mu$-harmonic mean $\lambda$ such that $\lambda(u) = \|u_\infty\|$. To prove this, we fix  a sequence $(g_n)$ such that
\[
\lim_n u(g_n) = \sup\big\{ u(g) \, : \, g \in G \big\},
\]
and define the sequence $(\lambda_m)$ of means on $G$ by
\[
\lambda_m(\phi) = \frac{1}{m} \sum_{n=1}^m \int_G \phi(xg_m) \, d\cmu^{*n}(x), \quad \phi \in \ell^\infty(G).
\]
Since $u$ is left $\mu$-harmonic, we have $\lambda_m(u) = u(g_m)$ for all $m$, and
one readily checks that any cluster point $\lambda$ of the sequence $(\lambda_m)$ in $M(G)$ is left $\mu$-harmonic and satisfies $\lambda(u) = \lim_m u(g_m)$. 

\section{Appendix I: Harmonic functions and affine isometric actions on Hilbert spaces}

As part of Theorem \ref{KL}, we proved that if $(G,\mu)$ is a measured Liouville group and $u$ is a left Lipschitz and left (quasi-)$\mu$-harmonic function on $G$, then $u$ must be a homomorphism. The aim of this 
appendix is to show that the combination "LEFT Lipschitz" and 
"LEFT quasi-$\mu$-harmonic" is crucial, and if one (but not both) is replaced by a "RIGHT", then the situation is quite different. Indeed,
we shall prove the following theorem, whose origin is hard to track down, but which is well-known to experts.

\begin{theorem}[Folklore]
Every infinite, finitely generated and symmetric measured group $(G,\mu)$,
where $\mu$ is assumed to be finitely supported,  
admits a non-trivial left Lipschitz and \emph{right} $\mu$-harmonic function.
\end{theorem}

Since every \emph{non-amenable} measured group $(G,\mu)$ admits a
wealth of non-trivial \emph{bounded} right $\mu$-harmonic functions,
the theorem is perhaps most interesting for amenable groups. However,
the construction which we will describe below works for a larger class of
groups, namely those which admit affine isometric actions on (real) Hilbert 
spaces with \emph{unbounded} orbits. It is well-known (see e.g. Theorem 
13.10 in \cite{Gla}) that this is equivalent to assuming that the group $G$ does 
\emph{not} have Kazhdan's Property (T). In particular, our construction will 
work for every countable (infinite) \emph{amenable} group. 

Recall that if $\cH$ is a real Hilbert space, then a map $T : \cH \ra \cH$ is 
an \emph{affine isometry} $T$ if it can be written on the form
\[
Tx = Ux + b, \quad x \in \cH,
\]
for some \emph{linear} isometry $U$ of $\cH$ and $b \in \cH$. Clearly, the
set of affine isometries of $\cH$ forms a group $\Aff(\cH)$ under composition,
and a homomorphism $\alpha : G \ra \Aff(\cH)$ is called an 
\emph{affine isometric action} of $G$ on $\cH$. Explicitly, we have
\[
\alpha(g)x = \pi(g)x + b(g)
\]
for some linear isometric representation $\pi$ of $G$ and a map $b : G \ra \cH$
which satisfies
\[
b(gh) = b(g) + \pi(g)b(h), \quad \forall \, g, h \in G.
\]
We shall refer to such maps as \emph{$\pi$-cocycles}, and we note that the 
action $\alpha$ has bounded orbits if and only if the corresponding $b$ is 
a norm-bounded function on $G$. 

\begin{proposition}
Let $(G,\mu)$ be a finitely generated and symmetric measured group, 
where $\mu$ is assumed to be finitely supported, and suppose $\alpha$ 
is an affine isometric action of $G$ on a \emph{real} Hilbert space $\cH$ 
without unbounded orbits. Then there exists $x_o$ and $y$ in $\cH$ such 
that 
\[
f(g) = \langle y, \alpha(g) \cdot x_o \rangle_\cH, \quad g \in G,
\]
is an \emph{unbounded}, left Lipschitz and \emph{right} $\mu$-harmonic function on $G$.
\end{proposition}

\begin{proof}[Sketch of the proof]
Recall that $\alpha$ can be written on the form
\[
\alpha(g)x = \pi(g)x + b(g),
\]
for some linear isometric representation $\pi$ of $G$ and a $\pi$-cocycle 
$b : G \ra \cH$. The assumption that the $\alpha$ has unbounded orbits 
simply means that
\[
\sup_{g \in G} \|\alpha(g)x\| = \infty, \quad \forall \, x \in \cH,
\]
and we shall prove that there exists $x_o \in \cH$ such that the orbit map
\[
F(g) = \alpha(g)x_o, \quad g \in G,
\]
satisfies $F * \mu = F$ in $\cH$, or equivalently (after some easy manipulations)
\begin{equation}
\label{eqharmonic}
\int_G \big( x_o - \alpha(s) x_o\big) \, d\mu(s) = 0.
\end{equation}
By the uniform boundedness principle, if 
\[
\sup_{g \in G} \big| \langle y, \alpha(g) x_o \rangle\big| < +\infty
\]
for all $y \in \cH$, then 
\[
\sup_{g \in G} \|\alpha(g) x_o\|< +\infty,
\]
which is a contradiction, and we conclude that there must exist 
$y \in \cH$ such that the function
\[
f(g) = \langle y, \alpha(g)x_o \rangle, \quad g \in G,
\]
is an \emph{unbounded} (and hence non-constant) right $\mu$-harmonic function 
on $G$. Also note that 
\[
\big|f(sg) - f(s)\big| = \big| \langle y,\pi(s)\big(\pi(g)x_o - x_o + b(g)\big)\rangle\big| \leq \|y\| \cdot \big(2 \|x_o\| + \|b(g)\|\big), 
\]
for all $g$ and $s$, which shows that $f$ is left Lipschitz. \\

To establish the existence of $x_o \in \cH$ such that \eqref{eqharmonic}
holds, we argue as follows. Consider the "energy functional"
\[
E(x) = \int_G \big\| \alpha(s) x - x\big\|^2 \, d\mu(s), \quad x \in \cH,
\]
which is well-defined since $\mu$ is finitely supported (but clearly this assumption can
be substantially weakened). One readily checks
that $E$ admits a local minimum $x_o$, and thus
\[
\frac{d}{dt}E(x_o + tv) \Big|_{t = 0} = 0, \quad \forall \, v \in \cH.
\]
The left hand side can be easily calculated. Indeed, after a series of calculations, using the assumptions that $\cH$ is a real Hilbert space and $\mu$ is a symmetric measure on $G$, we arrive at the identities, 
\[
\frac{d}{dt}E(x_o + tv) \Big|_{t = 0} 
= 
4 \cdot
\langle v, \int_G \big( x_o - \alpha(s) x_o\big) \, d\mu(s) \rangle = 0,
\]
for all $v \in \cH$, from which \eqref{eqharmonic} follows. 
\end{proof}

\section{Appendix II: Open problems and remarks}

We collect in this appendix some questions and remarks relating to the topics discussed
in this paper. 

\subsection{Drifts of random walks on homogeneous spaces}
Let $(G,\mu)$ be a countable measured group and denote by $(B,m)$ its Poisson boundary. Clearly, if $H < G$ is a subgroup which acts ergodically on $(B,m)$, then
there are no non-constant \emph{bounded} left $\mu$-harmonic functions on the quotient space $G/H$. However, it certainly also makes sense to ask whether 
\emph{unbounded} left $\mu$-harmonic functions can exist on the quotient space
$G/H$, at least when $H$ has infinite index in the group $G$. 

For instance, in the extreme case when $\mu$
is symmetric and finitely supported such that $(G,\mu)$ is a Liouville group (that is 
to say, $(B,m)$ is just a singleton space) and $H$ is the trivial subgroup, then the 
construction in Appendix I, shows that there are always unbounded left $\mu$-harmonic
functions. 

A less extreme case is suggested by Corollary \ref{CD}, which can be equivalently stated
as the assertion that there are no non-constant bounded left $\mu_o \otimes \cmu_o$-harmonic functions on the quotient $G_o \times G_o/\Delta_2 G_o$ for any measured group $(G_o,\mu_o)$. In this setting, the problem above can be equivalently 
formulated as follows. 

\begin{problem}
Does every measured group $(G,\mu)$ admit a non-trivial bi-$\mu$-harmonic (left and right $\mu$-harmonic) function?
\end{problem}

The problem for general quotient spaces seems intractable, and there could very well be obvious counter-examples. 

\begin{problem}
Construct a countable measured group $(G,\mu)$ and an infinite index subgroup $H < G$ such that the
quotient space $G/H$ does not admit \emph{any} non-constant left (quasi-)$\mu$-harmonic functions whatsoever. 
\end{problem}

It is clear that the notion of drift can be generalized to invariant metrics on more general
$G$-spaces (in particular coset spaces). An affirmative answer to the following question would generalize the Karlsson-Ledrappier Theorem (Theorem \ref{KL}) to this setting. 

\begin{problem}
Let $(G,\mu)$ be a measured group and $H < G$ a subgroup which acts ergodically on 
the Poisson boundary of $(G,\mu)$. If $d$ is a left $G$-invariant metric on the quotient
space $G/H$, is it then true that
\[
\lim_{n \ra \infty} \frac{1}{n} \sum_{k=1}^n \int_G d(gH,H) \, d\mu^{*k}(g) = 0?
\]
\end{problem}

One could start by analyzing the following special case which corresponds to the case when $G = G_o \times G_o$ and $H = \Delta_2 G_o$ and $\mu = \mu_o \otimes \mu_o$, for some countable group $G_o$ and some symmetric measure $\mu_o$ on $G_o$.

\begin{problem}
Let $(G,\mu)$ be a symmetric measured group and suppose there exists a \emph{bi-invariant} (conjugation-invariant) and $\mu$-integrable (semi-)metric $d$ on $G$. Is 
$\ell_d(\mu) = 0$?
\end{problem}

For instance, as a first test case, one could focus on the commutator subgroup $G$ of the free group on two generators and the stable commutator length on $G$.

\subsection{Harmonic Kronecker factors}

Let $G$ be a countable group and $(X,\nu)$ a non-singular ergodic $G$-space. Let 
$\cK$ denote the smallest $G$-invariant sub-$\sigma$-algebra of the Borel $\sigma$-algebra on $X$ with the property that $\cK \times \cK$ contains the $\sigma$-algebra
of all $G$-invariant subsets in $X \times X$. When $\nu$ is $G$-invariant, this $G$-invariant $\sigma$-algebra (or its corresponding factor) is usually called the 
\emph{Kronecker factor}, and it is a classical fact that the factor $G$-space is isomorphic
to an action by rotations of $G$ on a compact homogeneous space. Except for some 
remarks in \cite{AN87}, this factor does not seem to have attracted much attention, and
it seems hard to say anything significant about it in this generality. However, it 
could be that the situation for $(G,\mu)$-spaces is more amenable for a closer analysis.

\begin{problem}
Let $(G,\mu)$ be a symmetric measured group with Poisson boundary 
$(B,m)$ and suppose $(X,\nu)$ is an ergodic  $(G,\mu)$-space which admits 
$(B,m)$ as a factor. Assume that the product of $(X,\nu)$ with itself 
is \emph{not} ergodic. Does this mean that there exists a factor $(Y,\eta)$ 
of $(X,\nu)$ which is a non-trivial isometric extension of $(B,m)$? 
\end{problem}

Put differently, is $(Y,\eta)$ isomorphic (as a $G$-space) to a skew product 
of the form $(B \times K/K_o,m \otimes \eta)$, where $K$ is a compact group
and $K_o$ a closed subgroup and $\eta$ is the Haar probability measure on 
$K/K_o$, such that the $G$-action can be written as
\[
g(b,z) = (gb,c(g,b)z), \quad (b,z) \in B \times K/K_o,
\]
where $c : G \times B \ra K$ is a measurable cocycle?

\section{Acknowledgments}
The author would like to thank Vadim Kaimanovich for encouraging him 
to write up the present collection of random observations. His warm 
thanks also go to Uri Bader, Alex Furman, Yv\'es Guivarc'h, Yair Hartman,
Anders Karlsson, Gady Kozma and Amos Nevo 
for their never-ending willingness to discuss various problems relating to harmonic functions and $(G,\mu)$-spaces. Finally, several key insights in 
this paper are bi-products of on-going collaborations with Alexander Fish
and Tobias Hartnick, and the author wishes to express his gratitude to both of them.


\begin{thebibliography}{99}
\bibitem{AL05}
J. Aaronson, M. Lema\'nczyk,
\emph{Exactness of Rokhlin endomorphisms and weak mixing of Poisson boundaries.} Algebraic and topological dynamics, 77--87, Contemp. Math., \textbf{385}, Amer. Math. Soc., Providence, RI, 2005. 
\bibitem{AN87}
J. Aaronson, M. Nadkarni
\emph{$L^{\infty}$-eigenvalues and $L^2$-spectra of nonsingular transformations.} 
Proc. London Math. Soc. (3) \textbf{55} (1987), no. 3, 538--570.
\bibitem{A72}
A, Avez, 
\emph{Entropie des groupes de type fini}. 
C. R. Acad. Sci. Paris S\'er. A-B \textbf{275} (1972), 1363--1366.
\bibitem{Bj08}
M. Bj\"orklund, 
\emph{Central limit theorems for Gromov hyperbolic groups.} 
J. Theoret. Probab. \textbf{23} (2010), no. 3, 871--887.
\bibitem{BF3}
M. Bj\"orklund, A. Fish,
\emph{Product set phenomena in countable groups}. Preprint. 
\bibitem{BHM}
S. Blach\'ere, P. Ha\"issinsky, P. Mathieu. 
\emph{Harmonic measures versus quasiconformal measures for hyperbolic groups.} Ann. Sci. \'Ec. Norm. Sup\'er. 44, no. 4 (2011), 683 -- 721.
\bibitem{D80}
Y. Derriennic, 
\emph{Quelques applications du th\'eoreme ergodique sous-additif}, Ast\'erisque 74, (1980), 183--201.
\bibitem{EK10}
A. Erschler, A. Karlsson, \emph{Homomorphisms to $\bR$ constructed from random walks.} 
Ann. Inst. Fourier (Grenoble) 60 (2010), no. 6, 2095--2113.
\bibitem{Fu}
H, Furstenberg, \emph{Random walks and discrete subgroups of Lie groups.} 
1971 Advances in Probability and Related Topics, Vol. \textbf{1} pp. 1--63 Dekker, New York
\bibitem{FG10}
H. Furstenberg, E. Glasner,
\emph{Stationary dynamical systems}, 
Dynamical number--interplay between dynamical systems and number theory, 1--28, Contemp. Math., \textbf{532}, Amer. Math. Soc., Providence, RI, 2010.
\bibitem{Fo}
E. F\o lner, 
\emph{Note on a generalization of a theorem of Bogoliouboff.} 
Math. Scand. \textbf{2}, (1954). 224--226
\bibitem{Gla}
E. Glasner, \emph{Ergodic theory via joinings.} Mathematical Surveys and Monographs, \textbf{101}. American Mathematical Society, Providence, RI, 2003. xii+384 pp. ISBN: 0--8218--3372--3
\bibitem{GW13}
E. Glasner, B. Weiss,
\emph{Weak mixing properties for non-singular actions}. Preprint.
\bibitem{Gr}
M. Gromov, \emph{Hyperbolic manifolds, groups and actions}, pp. 183--213 in Riemann surfaces and related topics (Stony Brook, NY, 1978), edited by I. Kra and B. Maskit, Ann. of Math. Stud. \textbf{97}, Princeton Univ. Press, 1981.
\bibitem{HeHe}
H. Hennion, L. Herv\'e, 
\emph{Limit theorems for Markov chains and stochastic properties of dynamical systems by quasi-compactness.}
Lecture Notes in Mathematics, 1766. Springer--Verlag, Berlin, 2001. viii+145 pp. ISBN: 3-540-42415-6 
\bibitem{Ja94}
Jaworski, W. 
\emph{Strongly approximately transitive group actions, the Choquet-Deny theorem, and polynomial growth.} 
Pacific J. Math. 165 (1994), no. \textbf{1}, 115–129.
\bibitem{KV83}
V. Kaimanovich, A. Vershik, 
\emph{Random walks on discrete groups: boundary and entropy}, 
Ann. Prob. \textbf{11} (1983) 457--490
\bibitem{Ka92}
V. Kaimanovich, 
\emph{Bi-harmonic functions on groups.} (French summary) 
C. R. Acad. Sci. Paris S\'er. I Math. \textbf{314} (1992), no. 4, 259–264. 
\bibitem{K03}
V.A. Kaimanovich, \emph{Double ergodicity of the Poisson boundary and applications to bounded cohomology}. Geom. Funct. Anal. \textbf{13} (2003), no. 4, 852–861.
\bibitem{KL07}
A. Karlsson, F. Ledrappier, 
\emph{Linear drift and Poisson boundary for random walks.} 
Pure Appl. Math. Q. \textbf{3} (2007), no. 4, Special Issue: 
In honor of Grigory Margulis. Part 1, 1027--1036.
\bibitem{Kh35}
A. Khintchine,
\emph{Eine Versch\"arfung des Poincar\'eschen "Wiederkehrsatzes''.} (German) Compositio Math. \textbf{1} (1935), 177--179.
\bibitem{Ma}
G. W., Mackey, 
\emph{Ergodic transformation groups with a pure point spectrum.}
Illinois J. Math. \textbf{8} 1964 593–600. 
\bibitem{Mc74}
D. L. McLeish \emph{Dependent Central Limit Theorems and Invariance Principles} Ann. Prob. 2 (1974) 620--628.
\bibitem{Na1}
S.V. Nagaev,
\emph{Some limit theorems for stationary Markov chains.}
 (Russian) 
 Teor. Veroyatnost. i Primenen. 2 1957 389--416. 
\bibitem{Ra88}
A. Raugi, 
\emph{Un th\'eor\'eme de Choquet-Deny pour les groupes moyennables. (French. English summary) [A Choquet-Deny theorem for amenable groups]} 
Probab. Theory Related Fields \textbf{77} (1988), no. 4, 481–496. 
\bibitem{V85}
N. Th. Varopoulos, \emph{Long range estimates for Markov chains}, 
Bull. Sci. Math. \textbf{109} (1985) 225--252
\end{thebibliography}
\end{document}